\documentclass[graybox]{svmult}

\usepackage{mathptmx}       
\usepackage{helvet}         
\usepackage{courier}        
\usepackage{type1cm}        
\usepackage{makeidx}         
\usepackage{graphicx}        
\usepackage{multicol}        
\usepackage[bottom]{footmisc}
\usepackage{amsmath}
\usepackage{amssymb}
\usepackage{algorithm}
\usepackage{algpseudocode}
\usepackage{hyperref}

\newcommand{\R}{\mathbb{R}}

\newcommand{\Rext}{\mathbb{R}\cup\{+\infty\}}

\newcommand{\set}[1]{\left\{#1\right\}}
\newcommand{\norm}[1]{\Vert #1\Vert}

\newcommand{\argmin}{\mathrm{arg}\!\min}
\newcommand{\Eproof}{\hfill$\square$}
\newcommand*\widebar[1]{\@ifnextchar^{{\wide@bar{#1}{0}}}{\wide@bar{#1}{1}}}

\newif\ifshowproof
\showprooffalse 

\newcommand{\xb}{x}
\newcommand{\yb}{y}
\newcommand{\zb}{z}
\newcommand{\ub}{u}
\newcommand{\vb}{v}
\newcommand{\cb}{c}

\newcommand{\Xc}{\mathcal{X}}

\newcommand{\dom}[1]{\mathrm{dom}\left(#1\right)}

\newcommand{\Cc}{\mathcal{C}}
\newcommand{\Dc}{\mathcal{D}}
\newcommand{\Uc}{\mathcal{U}}

\newcommand{\Lc}{\mathcal{L}}

\newcommand{\wb}{w}
\newcommand{\Ab}{A}
\newcommand{\Bb}{B}

\newcommand{\iprods}[1]{\langle #1\rangle}

\newcommand{\prox}{\textrm{prox}}

\newcommand{\xopt}{x^{\star}}

\newcommand{\Xopt}{\mathcal{X}^{\star}}

\newcommand{\fopt}{f^{\star}}

\newcommand{\dopt}{d^{\star}}

\newcommand{\xbar}{\bar{x}}

\newcommand{\wbar}{\bar{w}}

\newcommand{\lbd}{\lambda}

\newcommand{\Ball}{\mathbb{B}}

\newcommand{\beforesec}{\vspace{-4ex}}
\newcommand{\aftersec}{\vspace{-2.5ex}}
\newcommand{\beforesubsec}{\vspace{-5ex}}
\newcommand{\aftersubsec}{\vspace{-2ex}}
\newcommand{\beforesubsubsec}{\vspace{-1.5ex}}
\newcommand{\aftersubsubsec}{\vspace{-2ex}}

\makeindex             

\begin{document}

\title*{Smooth Alternating Direction Methods for Fully Nonsmooth Constrained Convex Optimization}
\titlerunning{Smooth Alternating Optimization Methods}
\author{Quoc Tran-Dinh$^{\dagger*}$ and Volkan Cevher$^{\ddagger}$}
\institute{
$^{*}$Corresponding author.\vspace{1ex}\\
$^{\dagger}$Quoc Tran-Dinh \at
Department of Statistics and Operations Research,  
The University of North Carolina at Chapel Hill (UNC), 
333 Hanes Hall, UNC-Chapel Hill, NC, USA,
\email{quoctd@email.unc.edu}
\and
$^{\ddagger}$Volkan Cevher \at
Laboratory for Information and Inference Systems (LIONS),
\'{E}cole Polytechnique F\'{e}d\'{e}rale de Lausanne (EPFL), 
Station 11, Lausanne, Switzerland,
\email{ volkan.cevher@epfl.ch}
}
%
%
\maketitle

\abstract*{We propose two new alternating direction methods to solve ``fully'' nonsmooth constrained convex problems. 
Our algorithms have the best known  worst-case iteration-complexity guarantee under mild assumptions for both the objective residual and feasibility gap.
Through theoretical analysis, we show how to update all the algorithmic parameters automatically with clear impact on the convergence performance. 
We also provide a numerical illustration showing the advantages of our methods over the classical alternating direction methods using a well-known feasibility problem.
}

\abstract{We propose two new alternating direction methods to solve ``fully'' nonsmooth constrained convex problems. 
Our algorithms have the best known  worst-case iteration-complexity guarantee under mild assumptions for both the objective residual and feasibility gap.
Through theoretical analysis, we show how to update all the algorithmic parameters automatically with clear impact on the convergence performance. 
We also provide a representative numerical example showing the advantages of our methods over the classical alternating direction methods using a well-known feasibility problem.
}

\beforesec
\section{Introduction}\label{sec:intro}
\aftersec
In this paper, we aim at developing new optimization algorithms to solve nonsmooth constrained convex optimization problems of the form:
\begin{equation}\label{eq:constr_cvx}
\fopt := \left\{\begin{array}{ll}
\displaystyle\min_{\xb := (\ub, \vb)\in\R^p} &\big\{ f(\xb) := g(\ub) + h(\vb)  \big\},\\
\text{s.t.} & \Ab\ub + \Bb\vb = \cb,
\end{array}\right.
\vspace{-0.75ex}
\end{equation}
where $g: \R^{p_1} \to \Rext$ and $h : \R^{p_2}\to\Rext$ are two proper, closed and convex functions, $\Ab \in \R^{n\times p_1}$, $\Bb\in\R^{n\times p_2}$, $\cb\in\R^n$ are given, and $p := p_1 + p_2$. 
Although our proposed methods can solve \eqref{eq:constr_cvx} with both smooth and nonsmooth objective functions, we are more interested in the case where both $f$ and $g$ are nonsmooth.
In this case, we refer to \eqref{eq:constr_cvx} as a ``fully'' nonsmooth problem since, except for convexity,  we do not require any structure assumptions on $g$ and $h$ such as Lipschitz gradient or strong convexity.
Problem \eqref{eq:constr_cvx} covers many prominent applications such as convex feasibility problems \cite{Bauschke2011}, support vector machine \cite{Boyd2011}, matrix completion \cite{Candes2012b}, basis pursuit \cite{Simon2013}, among many others.

Associated with the primal problem \eqref{eq:constr_cvx}, we also look at the dual problem:
\vspace{-0.75ex}
\begin{equation}\label{eq:Fenchel_dual_prob}
d^{\star} := \min_{\lbd\in\R^n}\big\{ d(\lbd) :=  g^{*}(\Ab^{\top}\lbd) + h^{*}(\Bb^{\top}\lbd) - \iprods{\cb, \lbd}\big\},
\vspace{-0.75ex}
\end{equation}
where  $g^{*}$ and $h^{*}$ are the Fenchel conjugates \cite{Rockafellar1970} of $g$ and $h$, respectively; $d$ is the dual function; $\lbd$ is the dual variable; and $d^\star$ denotes the dual optimal value. 
The convex template \eqref{eq:constr_cvx} also manifests itself when we apply convex splitting techniques to decompose the composite objective $f$ into two terms $g$ and $h$ that are coupled via linear constraints. It can also include convex constraints on $\ub$ and $\vb$ via indicators.  

This paper develops a new primal-dual algorithmic framework to solve \eqref{eq:constr_cvx} which processes $g$ and $h$ in an alternating fashion to obtain approximate numerical solutions. 
The alternating optimization approach has regained popularity due to its ability to decentralize data, decompose problem components, and distribute computation in large-scale problems. 
The underlying theory for the classical alternating optimization methods, such as the alternating direction method of multipliers (ADMM) or the alternating minimization algorithm (AMA), is mature as they have their roots from the splitting methods in monotone inclusions and other classical approaches, such as forward-backward splitting, Douglas-Rachford splitting, Dykstra projections, and Hauzageau's methods \cite{alotaibi2014best,Bauschke2011}. 

Alternating optimization strategies often provide computational advantages as compared to processing both terms jointly. 
This approach leads to several methods and variants for solving \eqref{eq:constr_cvx} as can be found in the literature, see, e.g., \cite{bolte2014proximal,cai2016convergence,Chambolle2011,Davis2014,Davis2014b,Davis2015,Eckstein1992,Goldstein2012,He2012,Lin2013extragradient,lin2015iteration,lin2015global,Nemirovskii2004,Ouyang2014,Shefi2014,shefi2016rate,Tao2012a,Tseng1991,Wang2013a,Wei20131}.
Among those, ADMM and AMA are the most popular ones.
Unlike the standard AMA and ADMM methods and their variants mentioned here, we focus on the case that the objective functions $g$ and $h$ are nonsmooth and the sum $f$ does not have a ``tractable'' proximal operator.
As an example, in convex feasibility problems, we aim at finding a common point in the intersection of many convex sets.
This problem can be formulated into a nonsmooth constrained convex problem \eqref{eq:constr_cvx} as we are targeting here.
The ``full'' nonsmoothness of \eqref{eq:constr_cvx} creates some fundamental drawbacks for numerical algorithms.
First, algorithms that require gradients of the objective function are not applicable.
Second, evaluating a proximal operator of the full objective function $f$ becomes impractical.
Third, methods using penalty or augmented Lagrangian functions are often inefficient due to complicated subproblems and tuning parameters.
A more thorough discussion on our approach and existing methods is postponed to Section~\ref{sec:conclusion}.
In this paper, we overcome these drawbacks by proposing a combination of different techniques in optimization for solving \eqref{eq:constr_cvx}.

\vspace{1ex}
\noindent\textbf{Our contributions:}  
Our main contributions can be summarized as follows:
\begin{itemize}
\item[$\mathrm{(a)}$](\textit{Theory}) We introduce \emph{a split-gap reduction technique} as a new framework for deriving new alternating direction methods. 
Our framework unifies the model-based gap reduction technique of \cite{TranDinh2014b}, smoothing techniques, and the powerful forward-backward and Douglas-Rachford splitting techniques.  
We establish explicit relations between primal weighting strategy, the parameter choices, and the global convergence rate of the algorithms in our framework.

\item[$\mathrm{(b)}$](\textit{Algorithms and convergence guarantees}) 
We propose two new smoothing alternating direction optimization algorithms: smoothing alternating minimization algorithm (SAMA), and smoothing alternating direction method of multipliers (SADMM).
We derive update rules for all algorithmic parameters including penalty parameters in a heuristic-free fashion.  
We rigorously characterize  the convergence rate of our algorithms for both the objective residual $f(\bar{\xb}^k) - \fopt$ and the feasibility gap $\norm{\Ab\bar{\ub}^k + \Bb\bar{\vb}^k - \cb}$. 
To the best of our knowledge,  this is the best known global convergence rate that can be achieved under mildest assumptions in the literature.

\item[$\mathrm{(c)}$](\textit{Special cases})
We also illustrate that our technique can exploit additional assumptions on $\Ab$ or $\Bb$, $g$ and $h$, whenever they are available. 
\end{itemize}
Let us emphasize the following important points of our contribution.

(i)~(\textit{Mild assumptions})
We only assume that $g$ and $h$ are proper, closed, and convex, the solution set of \eqref{eq:constr_cvx} is nonempty, and Slater's condition holds.
We also require a technical assumption on the boundedness of the domain of $g$ and $h$. However, this assumption can be removed by using Lemma \ref{le:restricted_domain}.
Therefore, our methods can solve a broad class of convex optimization problems covered by \eqref{eq:constr_cvx}.

(ii)~(\textit{Computational complexity})
Our smooth AMA algorithm essentially has the same per-iteration complexity as  the standard AMA \cite{Tseng1991}.
Similarly, our smooth ADMM has essentially the same per-iteration complexity as the standard ADMM \cite{Boyd2011}.
Although we require additional computation for accelerated steps and averaging, this computation only requires vector-vector additions and scalar-vector multiplications, whose cost is negligible. 

(iii)~(\textit{Parameter update})
Our algorithms are heuristic-free in the sense that we update all the parameters automatically at each iteration including the so-called penalty parameter in alternating direction methods \cite{Beck2014,Necoara2008,Nesterov2005c}. 
This solves the major drawback in augmented Lagrangian-based methods.
We argue that this key feature is important in parallel and distributed implementation, when tuning parameters is impossible to carry out.
Intriguingly, our algorithms update their penalty parameters in a decreasing fashion in stark contrast to the classical algorithms.

(iv)~(\textit{Convergence guarantees})
The proposed methods achieve the best known global convergence rate on the primal problem \eqref{eq:constr_cvx} as well as on the dual one \eqref{eq:Fenchel_dual_prob} under required assumptions. 
Moreover, we can explicitly show how the choice of algorithmic parameters can trade-off the convergence guarantee of the objective residual $f(\bar{\xb}^k) - \fopt$  and the primal feasibility gap $\norm{\Ab\bar{\ub}^k + \Bb\bar{\vb}^k - \cb}$ in the worst case.

\vspace{1ex}
\noindent\textbf{Paper organization:} 
Section \ref{sec:primal_dual_form} briefly presents a primal-dual formulation of problem \eqref{eq:constr_cvx} under basic assumptions, and characterizes its optimality condition.
Section \ref{sec:smoothing} deals with a smoothing technique for the primal-dual gap function.
Section \ref{sec:new_ama} presents a smoothing AMA algorithm and analyzes its convergence. 
The strongly convex case is also studied in this section.
Section \ref{sec:new_admm} is devoted to developing a smoothing ADMM algorithm and analyzes its convergence. 
Section \ref{sec:num_experiments} presents numerical experiments to verify the performance of our algorithms.
We conclude with a discussion of our results in the context of existing work. 
For clarity of exposition, several technical and new proofs are moved to the appendix.

\vspace{1ex}
\noindent\textbf{Notation:} 
In the sequel, we refer to \eqref{eq:constr_cvx} as the primal problem.
We work on the real spaces $\R^p$ and $\R^n$, endowed with the inner product $\iprods{\xb, \lambda}$ and the standard Euclidean norm $\norm{\cdot}$.
We use the superscript $\top$ for both the transpose and adjoint operators.  
For a convex function $f$, we use $\partial{f}$ for its subdifferential, and $f^{*}$ for its Fenchel conjugate. 
For a convex set $\Xc$, we use $\delta_{\Xc}$ for its indicator function, and $\textrm{ri}(\Xc)$ for its relative interior.
We also use $\R_{++}$ for the set of positive real numbers.

For any proper, closed, and convex function $\varphi : \R^p\to\Rext$, the proximal operator is defined as follows:
\vspace{-0.75ex}
\begin{equation}\label{eq:prox}
\prox_{\varphi}(\xb) := \argmin_{\zb}\set{\varphi(\zb) + (1/2)\norm{\zb - \xb}^2}.
\vspace{-0.75ex}
\end{equation}
Generally, computing $\prox_{\varphi}$ is intractable. However, if $\prox_{\varphi}$ can be computed in a closed form efficiently or in polynomial time, then we say that $\varphi$ has a \textit{tractable} proximity operator. Several examples can be found, e.g., in \cite{Bauschke2011,Parikh2013}.

\beforesec
\section{Preliminaries: Lagrangian primal-dual formulation}\label{sec:primal_dual_form}
\aftersec
This section briefly describes the primal-dual formulation of \eqref{eq:constr_cvx} and our fundamental assumptions.

\beforesubsec
\subsection{The dual problem}
\aftersubsec
Let $\xb := (\ub, \vb) \equiv (\ub^{\top}, \vb^{\top})^{\top} \in \R^p$  be the primal variable, $\dom{f} := \dom{g}\times\dom{h}$, and  $\Dc :=  \set{ (\ub, \vb) \in \dom{f} ~\mid~ \Ab\ub + \Bb\vb = \cb }$ be the feasible set of \eqref{eq:constr_cvx}.
We define the Lagrange function of \eqref{eq:constr_cvx} associated with $\Ab\ub + \Bb\vb = \cb$ as $\Lc(\xb, \lbd) := g(\ub) + h(\vb) - \iprods{\lbd, \Ab\ub + \Bb\vb - \cb}$, where $\lbd\in\R^n$ is the Lagrange multiplier. 
We recall the dual problem \eqref{eq:Fenchel_dual_prob} of \eqref{eq:constr_cvx} here:
\begin{equation}\label{eq:dual_prob}
{\!\!\!\!}\dopt := \min_{\lbd\in\R^n}\set{ d(\lbd) := \max_{\ub}\set{\iprods{\Ab^{\top}\lbd, \ub} \!-\! g(\ub) } + \max_{\vb}\set{\iprods{\Bb^{\top}\lbd, \vb} \!-\! h(\vb)} \!-\! \cb^{\top}\lbd },{\!\!\!}
\end{equation}
where $d$ is the dual function, and two terms can individually  be computed as
\begin{equation}\label{eq:dual_func_i}
\left\{\begin{array}{lll}
\varphi(\lbd) &:= \displaystyle\max_{\ub\in\dom{g}}\set{ \iprods{\Ab^{\top}\lbd, \ub} - g(\ub) } &= g^{\ast}(\Ab^{\top}\lbd),\\
\psi(\lbd) &:= \displaystyle\max_{\vb\in\dom{h}}\set{ \iprods{\Bb^{\top}\lbd, \vb} - h(\vb) } - \cb^{\top}\lbd &= h^{\ast}(\Bb^{\top}\lbd) - \cb^{\top}\lbd.
\end{array}\right.
\end{equation} 
Let us denote by $\ub^{\ast}(\lbd)$ and $\vb^{\ast}(\lbd)$ one solution of these subproblems, respectively, if they exist. 
In this case, using the optimality condition, we have $\Ab^{\top}\lbd \in\partial{g}(\ub^{\ast}(\lbd))$, which is equivalent to $\ub^{\ast}(\lbd)\in\partial{g^{\ast}}(\Ab^{\top}\lbd)$.
Similarly, $\Bb^{\top}\lbd \in\partial{h}(\vb^{\ast}(\lbd))$, which is equivalent to $\vb^{\ast}(\lbd) \in\partial{h^{\ast}}(\Bb^{\top}\lbd)$.
These dual components are convex, but generally nonsmooth. 
Subgradient or bundle-type methods for directly solving \eqref{eq:dual_prob} are generally inefficient \cite{Nemirovskii1983,Nesterov2004}.

\beforesubsec
\subsection{Basic assumptions}
\aftersubsec
Let us denote by $\Xopt$ the solution set of \eqref{eq:constr_cvx}.
We say that the \textit{Slater condition} holds for \eqref{eq:constr_cvx} if we have
\begin{equation}\label{eq:slater_cond}
\mathrm{ri}(\dom{f})\cap\set{(\ub,\vb) \in\R^p \mid \Ab\ub + \Bb\vb = \cb} \neq\emptyset,
\end{equation}
where $\mathrm{ri}(\Xc)$ is the relative interior of $\Xc$ (see \cite{Rockafellar1970}).

For  the primal-dual pair \eqref{eq:constr_cvx} and \eqref{eq:dual_prob}, we require the following assumption:

\begin{assumption}\label{as:A1}
The functions $g$ and $h$ are proper, closed, and convex.
The solution set $\Xopt$ of \eqref{eq:constr_cvx} is nonempty. 
Either $\dom{f}$ is polyhedral or the Slater condition \eqref{eq:slater_cond} holds.
\end{assumption}
Compared to existing methods in the literature \cite{bolte2014proximal,cai2016convergence,Chambolle2011,Davis2014,Davis2014b,Davis2015,Eckstein1992,Goldstein2012,He2012,Lin2013extragradient,lin2015iteration,lin2015global,Nemirovskii2004,Ouyang2014,Shefi2014,shefi2016rate,Tao2012a,Tseng1991,Wang2013a,Wei20131}, this assumption is perhaps the mildest one so far.
We do not require any strong convexity, error bound, regularity, or Lipschitz gradient assumptions on $g$ and $h$.

\vspace{-1ex}
\beforesubsec
\subsection{Zero duality gap}
\vspace{-1ex}
\aftersubsec
Under Assumption~\ref{as:A1}, the solution set $\Lambda^{\star}$ of the dual problem \eqref{eq:dual_prob}  is nonempty and  bounded.
Moreover,  \textit{strong duality} holds, i.e., $\fopt  + \dopt = 0$. 
From  the classical duality theory, we have $f(\xb) + d(\lbd) \geq 0$ for any feasible primal-dual point $(\xb, \lbd)$. 
Hence, the duality gap function $G$ is defined by
\vspace{-0.75ex}
\begin{equation}\label{eq:gap_func}
G(\wb) := f(\xb) + d(\lbd) \geq 0, ~~\forall\xb \in \mathcal{D}, ~\forall\lbd\in\R^n,
\vspace{-0.75ex}
\end{equation}
where $\wb := (\xb, \lbd)$. 
Clearly, $G(\wb^{\star}) = 0$ (zero duality gap) for any primal-dual solution $\wb^{\star} := (\xopt, \lbd^{\star}) \in \Xopt\times\Lambda^{\star}$.
In addition, $\wb^{\star}$ is a saddle point of the Lagrange function; that is $\Lc(\xopt,\lbd) \leq \Lc(\xopt,\lbd^{\star}) = \fopt = -d^{\star} \leq \Lc(\xb, \lbd^{\star})$ for all $\xb\in\dom{f}$ and $\lbd\in\R^n$.
The optimality condition of \eqref{eq:constr_cvx} can be written as
\vspace{-0.75ex}
\begin{equation}\label{eq:opt_cond}
\Ab\ub^{\star} + \Bb\vb^{\star} = \cb, ~~\Ab^{\top}\lbd^{\star} \in \partial{g}(\ub^{\star}),~~\text{and}~~\Bb^{\top}\lbd^{\star} \in \partial{h}(\vb^{\star}).
\end{equation}

\vspace{-2ex}
\beforesubsec
\subsection{Technical assumption}
\vspace{-1ex}
\aftersubsec
Apart from Assumption~\ref{as:A1}, the methods we will develop in the following sections require the following  boundedness assumption:

\begin{assumption}\label{as:A2}
Both $\dom{g}$ and $\dom{h}$ are bounded.
\end{assumption}

According to \cite[Corollary 17.19]{Bauschke2011}, the boundedness of $\dom{g}$ and $\dom{h}$ is equivalent to the Lipschitz continuity of the conjugates $g^{\ast}$ and $h^{\ast}$, respectively.
Assumption \ref{as:A2} also theoretically restricts the class of problems in \eqref{eq:constr_cvx} that we can solve.  
However, if Assumption~\ref{as:A2} does not hold, then we can always add an artificial constraint $\norm{x} \leq R$ to \eqref{eq:constr_cvx} (or $\norm{u} \leq R$ and $\norm{v} \leq R$) so that Assumption~\ref{as:A2} is satisfied for this modified problem, where $R\in (0, +\infty)$. 
Under a proper choice of $R$, this problem is equivalent to \eqref{eq:constr_cvx} as showed in the following lemma.

\begin{lemma}\label{le:restricted_domain}
Consider two constrained convex optimization problems:
\begin{equation*} 
(\mathrm{P}_{\infty})~~f^{\star} := \min_{x\in\Dc} f(x)~~~~~\textrm{and}~~~~~(\mathrm{P}_R)~~\bar{f}^{\star} := \min_{x\in\Dc}\set{ f(x) \mid \norm{x} \leq R},
\end{equation*}
where $f$ is defined in \eqref{eq:constr_cvx}, $\Dc := \set{x = (u, v) \mid Au + Bv = c, ~u\in\dom{g}, ~v\in\dom{h}}$ is the feasible set of \eqref{eq:constr_cvx}, and $R\in (0, +\infty)$.

If $x^{\star}$ is a solution of $(\mathrm{P}_{\infty})$, and $\norm{x^{\star}} \leq R$, then it is a solution of $(\mathrm{P}_R)$.
Conversely, if $\bar{x}^{\star}$ is a solution of $(\mathrm{P}_R)$ and $\norm{\bar{x}^{\star}} < R$, then it is a solution of $(\mathrm{P}_{\infty})$.
\end{lemma}

\begin{proof}
It is obvious that if $x^{\star}$ is a solution of ($\mathrm{P}_{\infty}$), and $\norm{x^{\star}} \leq R$, then it is a solution of ($\mathrm{P}_R$).
Conversely, if $\bar{x}^{\star}$ is a solution of ($\mathrm{P}_R$), then we have $f(\bar{x}^{\star}) \leq f(x)$ for all $x\in\Dc$ and $\norm{x} \leq R$.
Take any $x\in\Dc\backslash\mathbb{B}_{R}$, where $\mathbb{B}_{R} := \set{x\in\R^p\mid \norm{x} \leq R}$ is a ball centered at the origin with radius $R$.
Since $\bar{x}^{\star} \in \mathrm{int}(\mathbb{B}_{R})$, the interior of $\mathbb{B}_{R}$, there exists $\hat{x}$ on the open segment $(\bar{x}^{\star}, x)$ such that $\hat{x} = (1-\tau)\bar{x}^{\star} + \tau x$ and $\hat{x} \in \Dc\cap\mathbb{B}_{R}$, where $\tau \in (0, 1)$.
In this case, by convexity of $\varphi$, we have $f(\bar{x}^{\star}) \leq f(\hat{x}) = f(1-\tau)\bar{x}^{\star} + \tau x) \leq (1-\tau)f(x^{\star}) + \tau f(x)$.
Since $\tau\in (0, 1)$, this inequality implies $f(\bar{x}^{\star}) \leq f(x)$. Therefore, $\bar{x}^{\star}$ is a solution of ($\mathrm{P}_{\infty}$).
\Eproof
\end{proof}

As suggested by Lemma \ref{le:restricted_domain}, if we add artificial bounds $\norm{u} \leq R$ and $\norm{v} \leq R$ to \eqref{eq:constr_cvx}, then the resulting problem is equivalent to
\begin{equation*}
\min_{u, v}\set{ \hat{g}(u) + \hat{h}(v) \mid \Ab\ub + \Bb\vb = \cb},
\end{equation*}
where $\hat{g} := g + \delta_{\Ball_R}$, $\hat{h} := h + \delta_{\Ball_R}$, and $\delta_{\Ball_R}$ is the indicator function of the closed ball $\Ball_R := \set{ z \mid \norm{z} \leq R}$.
This problem is of the same form as \eqref{eq:constr_cvx}.
Under Assumption \ref{as:A2}, the following quantity:
\begin{equation}\label{eq:Df_diameter}
D_f := \sup_{\ub\in\dom{g},~ \hat{\vb},\vb\in\dom{h}}\Big\{ \max\left\{ \norm{\Ab\ub + \Bb\vb - \cb}, \norm{\Ab\ub + \Bb(2\hat{v} - \vb) - \cb}\right\} \Big\}
\end{equation}
is bounded, i.e., $0 \leq D_f < +\infty$.

Note that, in our algorithms below, since we do not require $D_f$ as an input of the algorithms, this quantity can heuristically be estimated after we terminate the algorithms, and estimate the corresponding artificial radius $R$ based on iteration sequences obtained from the algorithms (see Remark \ref{re:artificial_bound}).

\beforesubsec
\section{Smoothing the primal-dual gap function}\label{sec:smoothing}
\aftersubsec
The dual function $d$ defined by \eqref{eq:dual_prob} is convex, but it is generally nonsmooth.
Our key idea is to replace the component $g^{\ast}$ in \eqref{eq:dual_func_i} with a new smoothed approximation $g^{\ast}_\gamma$ to derive new algorithms.

Let us consider the domain  $\Uc := \dom{g}$ of $g$. 
Associated with $\Uc$, we choose a proximity function $\omega$, i.e., $\omega$ is continuous and strongly convex with the convexity parameter $\mu_{\omega} = 1 > 0$, and $\Uc\subseteq\dom{\omega}$. In addition, we assume that $\omega$ is smooth, and its gradient is Lipschitz continuous with the Lipschitz constant $L_{\omega} \in [0, +\infty)$.

Given $\omega$, we define the associated Bregman distance
\begin{equation}\label{eq:bregman_dist}
b_{\Uc}(\ub, \hat{\ub}) := \omega(\ub) - \omega(\hat{\ub}) - \iprods{\nabla{\omega}(\hat{\ub}), \ub - \hat{\ub}}.
\end{equation}
Let $\bar{u}_c := \argmin_{u}\omega(u)$ be the prox-center of $\omega$, which exists and is unique. 
We consider the function $b_{\Uc}(\cdot,\bar{\ub}^c)$.
Clearly, $b_{\Uc}(\cdot,\bar{\ub}^c)$ is smooth and strongly convex with the convexity parameter $\mu_{b} = \mu_{\omega} = 1$. 
Its gradient $\nabla_1{b_{\Uc}(\ub,\bar{\ub}^c)} = \nabla{\omega}(\ub) - \nabla{\omega}(\bar{\ub}^c)$ is Lipschitz continuous with the Lipschitz constant $L_b = L_{\omega} \geq \mu_{\omega} = 1$. In addition, $b_{\Uc}(\bar{\ub}^c,\bar{\ub}^c) = 0$ and $\nabla_1{b_{\Uc}(\bar{\ub}^c, \bar{\ub}^c)} = 0$.

Given $b_{\Uc}(\cdot, \bar{\ub}^c)$, and  the conjugate $g^{\ast}$ of $g$, we define
\begin{equation}\label{eq:d1_gamma}
g^{\ast}_{\gamma}(\zb) := \max_{\ub\in\R^{p_1}}\set{ \iprods{\zb, \ub} - g(\ub) - \gamma b_{\Uc}(\ub,\bar{\ub}^c) },
\end{equation}
where $\gamma > 0$ is a smoothness parameter. 
We denote by $\ub^{\ast}_{\gamma}(\zb)$ the solution of the maximization problem in \eqref{eq:d1_gamma}, i.e.:
\vspace{-0.75ex}
\begin{equation}\label{eq:u_ast_gamma}
\ub^{\ast}_{\gamma}(\zb) := \mathrm{arg}\max_{\ub\in\R^{p_1}}\set{ \iprods{\zb, \ub} - g(\ub) - \gamma b_{\Uc}(\ub,\bar{\ub}^c) },
\vspace{-0.75ex}
\end{equation}
which is well-defined and unique.
Clearly, $\nabla{g_{\gamma}^{\ast}}(\zb) = \ub^{\ast}_{\gamma}(\zb)$ is the gradient of $g_{\gamma}^{\ast}$, which has $(1/\gamma)$-Lipschitz gradient. Hence, $g^{\ast}_{\gamma}$ is $(1/\gamma)$-smooth \cite{Bauschke2011}.

Let $g_{\gamma}^{\ast}$ and $\psi$ be defined by \eqref{eq:d1_gamma} and \eqref{eq:dual_func_i}, respectively, and $\beta > 0$.
We consider  
\begin{equation}\label{eq:G_gammabeta}
\left\{\begin{array}{ll}
d_{\gamma}(\lbd)  &:= g_{\gamma}^{\ast}(\Ab^{\top}\lbd) + \left( h^{\ast}(\Bb^{\top}\lbd) - \iprods{\cb, \lbd}\right) = \varphi_{\gamma}(\lbd) +  \psi(\lbd),\vspace{1ex}\\
f_{\beta}(\xb) &:= g(\ub) + h(\vb) + \frac{1}{2\beta}\norm{\Ab\ub + \Bb\vb - \cb}^2,\vspace{1ex}\\
G_{\gamma\beta}(\wb) &:= f_{\beta}(\xb) + d_{\gamma}(\lbd).
\end{array}\right.
\end{equation}
If $\gamma\downarrow 0^{+}$, then we have $d_{\gamma}(\lbd) \to d(\lbd)$. Hence, $d_{\gamma}$ is a smoothed approximation of $d$, but it is not fully smooth due to possible nonsmoothness of $\psi$.
For any feasible point $\xb = (\ub, \vb) \in \mathcal{D}$, we have $f_{\beta}(\xb) = f(\xb)$.  
Here, $f_{\beta}$ can be considered as an approximation to $f$ near the feasible set $\mathcal{D}$.
Hence, the smoothed gap function $G_{\gamma\beta}$ is an approximation of the duality gap function $G$ in \eqref{eq:gap_func}.
Moreover, the smoothed gap function $G_{\gamma\beta}$ is convex.
The following lemma shows us how to use  $G_{\gamma\beta}$ to characterize the primal-dual solutions for \eqref{eq:constr_cvx}-\eqref{eq:Fenchel_dual_prob}, whose proof is in Appendix \ref{apdx:le:optimal_bounds}.

\begin{lemma}\label{le:optimal_bounds}
For any $\bar{\xb}^k := (\bar{\ub}^k, \bar{\vb}^k)\in\dom{f}$ and $\bar{\lbd}^k\in\R^n$, it holds that
\begin{equation}\label{eq:lower_bound}
-\norm{\lbd^{\star}}\norm{\Ab\bar{\ub}^k + \Bb\bar{\vb}^k - \cb} \leq f(\xbar^k) - \fopt \leq f(\xbar^k) + d(\bar{\lbd}^k).
\end{equation}
Let $\{ \bar{\wb}^k\}$ be an arbitrary sequence in $\dom{f}\times\R^n$ and $\{(\gamma_k,\beta_k)\}$ be a sequence in $\R^2_{++}$. 
Then, the following estimates hold:
\begin{equation}\label{eq:optimal_bounds}
\left\{\begin{array}{ll}
f(\bar{\xb}^k) - \fopt &\leq S_k(\wbar^k), \vspace{1ex}\\
\norm{\Ab\bar{\ub}^k + \Bb\bar{\vb}^k - \cb} & \leq  2\beta_k\norm{\lambda^{\star}} + \sqrt{2\beta_kS_k(\wbar^k)}, \vspace{1ex}\\
d(\bar{\lbd}^k) - d^{\star} &\leq 2\beta_k\norm{\lbd^{\star}}^2 + \norm{\lbd^{\star}}\sqrt{2\beta_kS_k(\wbar^k)} + S_k(\wbar^k),
\end{array}\right.
\end{equation}
where  $S_k(\wbar^k) := G_{\gamma_k\beta_k}(\bar{\wb}^k) + \gamma_kb_{\Uc}(\ub^{\star},\bar{\ub}^c)$, which requires the values of $G_{\gamma\beta}$.
\end{lemma}

Computing exactly a primal-dual solution $(\xopt, \lbd^{\star})$ is impractical. 
Hence, our objective is to find  an approximation $(\bar{\xb}^k, \bar{\lbd}^k)$ to $(\xopt, \lbd^{\star})$ in the following sense:

\begin{definition}\label{de:approx_sol}
Given an accuracy $\varepsilon > 0$, a primal-dual point $(\bar{\xb}^k, \bar{\lbd}^k)\in \dom{f}\times\R^n$ is said to be an $\varepsilon$-solution of \eqref{eq:constr_cvx}-\eqref{eq:Fenchel_dual_prob} if 
\begin{equation*}
f(\bar{\xb}^k) - \fopt \leq \varepsilon, ~~~~\norm{\Ab\bar{\ub}^k + \Bb\bar{\vb}^k - \cb} \leq \varepsilon, ~~~\text{and}~~~d(\bar{\lambda}^k) - d^{\star} \leq \varepsilon.
\end{equation*}
We use the same accuracy parameter $\epsilon$ for each of these terms for simplicity. 
\end{definition}

We note that by combining $\norm{\Ab\bar{\ub}^k + \Bb\bar{\vb}^k - \cb} \leq \varepsilon$ and \eqref{eq:lower_bound}, we can guarantee a lower abound $f(\xbar^k) - \fopt \geq -\norm{\lambda^{\star}}\varepsilon$.
In addition, the domain $\dom{f}$ is usually simple (e.g., box, ball, cone, or simplex) so that the constraint $\bar{\xb}^k\in\dom{f}$ can be guaranteed via a closed form projection onto $\dom{f}$.

The goal is to generate a primal-dual sequence $\{\bar{\wb}^k\}$ and a parameter sequence $\{(\gamma_k, \beta_k)\}$ in Lemma \ref{le:optimal_bounds} such that $\{ G_{\gamma_k\beta_k}(\bar{\wb}^k)\}$ converges to $0$ and $\{(\gamma_k, \beta_k)\}$ also converges to zero.
Moreover, the convergence rate of $f(\bar{\xb}^k) - \fopt$ and $\norm{\Ab\bar{\ub}^k + \Bb\bar{\vb}^k - \cb}$ depends on the convergence rate of $\{ G_{\gamma_k\beta_k}(\bar{\wb}^k)\}$ and $\{(\gamma_k, \beta_k)\}$.

\beforesec
\section{Smoothing Alternating Minimization Algorithm \eqref{eq:new_ama_alg}}\label{sec:new_ama}
\aftersec
We propose a new alternating direction method via the application of the accelerated forward-backward splitting to the smoothed gap function. 
We describe \ref{eq:new_ama_alg} in three subsections: main steps, initialization, and parameter updates.

\beforesubsec
\subsection{Main steps}
\aftersubsec
At the iteration $k\geq 0$, given $\hat{\lbd}^k\in\R^n$ and the parameters $\gamma_{k\!+\!1} > 0$ and $\eta_k > 0$, the main steps of our \ref{eq:new_ama_alg} consists of two primal alternating direction steps and one dual ascend step as follows:
\begin{equation}\label{eq:new_ama_alg}
\left\{\begin{array}{ll}
\hat{\ub}^{k\!+\!1} &  := \mathrm{arg}{\!\!\!\!\!\!}\displaystyle\min_{\ub \in\dom{g}} \big\{ g(\ub) - \iprods{\Ab^{\top}\hat{\lbd}^k, \ub}  + \gamma_{k\!+\!1}b_{\Uc}(\ub,\bar{\ub}^c)  \big\},\vspace{0.5ex}\\
\hat{\vb}^{k\!+\!1} & := \mathrm{arg}{\!\!\!\!\!\!}\displaystyle\min_{\vb \in\dom{h}} \big\{ h(\vb)  -  \iprods{\Bb^{\top}\hat{\lbd}^k, \vb}  + \frac{\eta_k}{2}\Vert\Ab\hat{\ub}^{k\!+\!1} + \Bb\vb - \cb\Vert^2\big\},\vspace{0.5ex}\\
\bar{\lbd}^{k\!+\!1} &:= \hat{\lbd}^k - \eta_k(\Ab\hat{\ub}^{k\!+\!1} + \Bb\hat{\vb}^{k\!+\!1} - \cb),
\end{array}\right.
\tag{$\mathrm{SAMA}$}
\end{equation}
where $\gamma_{k\!+\!1}$ and $\eta_k$ are referred to as the smoothness and the penalty parameter, respectively, and $\bar{\ub}_c$ is the prox-center of $\omega$ in \eqref{eq:bregman_dist}. 

The subproblems in \ref{eq:new_ama_alg} can often be computed in a closed form. Let us describe two cases. First, if $b_{\Uc}(\cdot, \bar{\ub}^c) := (1/2)\norm{\cdot - \bar{\ub}^c}^2$, the standard Euclidean distance, then computing $\hat{\ub}^{k\!+\!1}$ reduces to computing the proximal operator of $g$, i.e., 
\begin{equation*}
\hat{\ub}^{k\!+\!1} = \prox_{\gamma_{k\!+\!1}^{-1}g}\big( \bar{\ub}_c + \gamma_{k\!+\!1}^{-1}\Ab^{\top}\hat{\lbd}^k\big).
\end{equation*} 
Second, if we have $\Bb=\mathbb{I}$ or $\Bb$ is orthonormal, then computing $\hat{\vb}^{k\!+\!1}$ reduces to computing the proximal operator of $h$, i.e., 
\begin{equation*}
\hat{\vb}^{k\!+\!1} = \prox_{\eta_k^{-1}h}\big(\Bb^{\top}(\cb - \Ab\hat{\ub}^{k\!+\!1}) + \eta_k^{-1}\Bb^{\top}\hat{\lbd}^k\big).
\end{equation*}

By inspection, it is easy to see that \ref{eq:new_ama_alg} is an analog of the classical AMA (cf., \eqref{eq:ama_scheme}). The  first subproblem, due to \eqref{eq:d1_gamma}, corresponds to the forward step while the last two lines correspond to the backward step. 
Moreover, if we set $\gamma_{k\!+\!1} = 0$ and $\hat{\lbd}^{k\!+\!1} = \bar{\lbd}^{k\!+\!1}$, \ref{eq:new_ama_alg} becomes  AMA. 
However, in contrast to the AMA, the \ref{eq:new_ama_alg} also features a dual acceleration and a primal weighted averaging step:
\begin{equation}\label{eq:add_ama_steps}
{\!\!\!\!}\left\{\begin{array}{llll}
&\hat{\lbd}_k &:= (1-\tau_k)\bar{\lbd}^k + \tau_k\lbd^{\ast}_k, &\text{(dual acceleration)}\vspace{0.75ex}\\
&(\bar{\ub}^{k\!+\!1}, \bar{\vb}^{k\!+\!1}) &:= (1-\tau_k)(\bar{\ub}^{k}, \bar{\vb}^{k}) + \tau_k(\hat{\ub}^{k\!+\!1}, \hat{\vb}^{k\!+\!1}), &\text{(weighted averaging)}
\end{array}\right.{\!\!\!\!}
\end{equation}
where $\lbd^{\ast}_k := \beta^{-1}_k(\cb - \Ab\bar{\ub}^k - \Bb\bar{\vb}^k)$, and $\tau_k \in (0, 1)$ is a given step size.
As we will prove in Theorem \ref{th:new_ama} below, these dual acceleration and primal weighted averaging steps allow us to achieve a better convergence rate on both the primal and the dual spaces compared to standard AMA methods \cite{Goldstein2012}.

The following lemma provides conditions showing that the sequence $\{(\xbar^k,\bar{\lbd}^k)\}$ generated by \eqref{eq:new_ama_alg}-\eqref{eq:add_ama_steps} maintains the non-monotone gap reduction condition introduced in  \cite{TranDinh2015b}. The proof of this lemma can be found in Appendix \ref{apdx:le:ama_key_estimate2}.

\begin{lemma}\label{le:ama_key_estimate2}
Let $\{\bar{\wb}^k \}$ with $\bar{\wb}^k := (\bar{\ub}^k, \bar{\vb}^k, \bar{\lbd}^k)$ be the sequence generated by \eqref{eq:new_ama_alg}-\eqref{eq:add_ama_steps}.
If  $\tau_k \in (0, 1]$ and  $\gamma_k, \beta_k,  \eta_k \in \R_{++}$ satisfy the following conditions:
\begin{equation}\label{eq:ama_param_cond}
\begin{array}{lll}
&(1+L_b^{-1}\tau_k)\gamma_{k\!+\!1} \geq  \gamma_k, ~~&\beta_{k\!+\!1} \geq (1-\tau_k)\beta_k, \vspace{1ex}\\
&(1-\tau_k^2)\gamma_{k\!+\!1}\beta_{k} \geq 2\norm{\Ab}^2\tau_k^2,~~\text{and}~~&2\norm{\Ab}^2\eta_k = \gamma_{k\!+\!1},
\end{array}
\end{equation}
then the following non-monotone gap reduction condition holds:
\begin{equation}\label{eq:ama_key_estimate2}
G_{\gamma_{k\!+\!1}\beta_{k\!+\!1}}(\bar{\wb}^{k\!+\!1}) \leq (1-\tau_k)G_{\gamma_k\beta_k}(\bar{\wb}^k)  + \frac{\eta_k\tau^2_k}{4}D_f^2,
\end{equation}
where $G_{\gamma_k\beta_k}$ is defined by \eqref{eq:G_gammabeta} and $D_f$ is defined by \eqref{eq:Df_diameter}.
\end{lemma}

\beforesubsec
\subsection{Initialization}
\aftersubsec
We note that we can initialize the algorithm at any starting point $\bar{\wb}^1 := (\bar{\ub}^1, \bar{\vb}^1, \bar{\lbd}^1)$. 
However, the convergence bounds will depend on $G_{\gamma_1\beta_1}(\bar{\wb}^1)$.
In order to provide transparent convergence results, we propose to use the following initialization in Lemma~\ref{le:init_point}, whose proof is given in Appendix \ref{apdx:le:init_point}.

\begin{lemma}\label{le:init_point}
Given $\hat{\lbd}^0 \in\R^m$, $\gamma_1 > 0$, and $\eta_0 > 0$, let $(\bar{\ub}^1, \bar{\vb}^1, \bar{\lbd}^1)$ be computed by 
\begin{equation}\label{eq:init_point}
\left\{\begin{array}{ll}
\bar{\ub}^1 &  := \mathrm{arg}{\!\!\!\!\!\!}\displaystyle\min_{\ub \in\dom{g}} \big\{ g(\ub) - \iprods{\Ab^{\top}\hat{\lbd}^0, \ub}  + \gamma_1b_{\Uc}(\ub,\bar{\ub}^c)  \big\},\vspace{0.5ex}\\
\bar{\vb}^1 & := \mathrm{arg}{\!\!\!\!\!\!}\displaystyle\min_{\vb \in\dom{h}} \big\{ h(\vb)  -  \iprods{\Bb^{\top}\hat{\lbd}^0, \vb}  + \frac{\eta_0}{2}\Vert\Ab\bar{\ub}^1 + \Bb\vb - \cb\Vert^2\big\},\vspace{0.5ex}\\
\bar{\lbd}^1 &:= \hat{\lbd}^0 - \eta_0(\Ab\bar{\ub}^1 + \Bb\bar{\vb}^1 - \cb).
\end{array}\right.
\end{equation}
Then, for any $\beta_1 > 0$, $\bar{\wb}^1 := (\bar{\ub}^1, \bar{\vb}^1, \bar{\lbd}^1)$, and $G_{\gamma\beta}$ defined by \eqref{eq:G_gammabeta} satisfy
\begin{align}\label{eq:init_point_est}
\begin{array}{ll}
{\!\!\!}G_{\gamma_1\beta_1}(\bar{\wb}^1) & \leq  \frac{\eta_0}{4}D_f^2 \!+\!  \frac{1}{2\eta_0^2}\Big[\frac{1}{\beta_1} \!-\! \frac{(5\gamma_1  \!\!-\! 2\eta_0\norm{\Ab}^2)\eta_0}{2\gamma_1}\Big]\norm{\bar{\lambda}^1-\hat{\lambda}^0}^2  \vspace{1ex}\\
& +\eta_0^{-1}\iprods{\hat{\lambda}^0, \bar{\lambda}^1 \!\!-\! \hat{\lambda}^0}.{\!\!\!\!}
\end{array}
\end{align}
Consequently, if we choose $\gamma_1$, $\beta_1$, and $\eta_0$ such that $5\gamma_1 > 2\eta_0\norm{\Ab}^2$ and $\beta_1 \geq \frac{2\gamma_1}{(5\gamma_1 - 2\eta_0\norm{\Ab}^2)\eta_0}$, then $G_{\gamma_1\beta_1}(\bar{\wb}^1) \leq \frac{\eta_0}{4} D_f^2 + \eta_0^{-1}\iprods{\hat{\lambda}^0, \bar{\lambda}^1 - \hat{\lambda}^0}$.
\end{lemma}

\beforesubsec
\subsection{Updating the parameters}
\aftersubsec
For simplicity of presentation, we choose $\omega$ as $\omega(\ub) := \frac{1}{2}\norm{\ub - \bar{\ub}^c}^2$ for a fixed $\bar{\ub}_c\in\dom{g}$.
In this case, $b_{\Uc}(\cdot,\bar{u}_c)$ defined by \eqref{eq:bregman_dist} becomes $b_{\Uc}(\cdot,\bar{u}_c) = \frac{1}{2}\norm{\cdot - \bar{u}_c}^2$.
Hence, we can update $\tau_k, \gamma_k$, $\beta_k$ and $\eta_k$ such that the equality in the conditions \eqref{eq:ama_param_cond} holds.
The following lemma provides \textbf{one possibility} to update these parameters whose proof is given in Appendix~\ref{apdx:le:update_rules_ama}.

\begin{lemma}\label{le:update_rules_ama}
Let $b_{\Uc}$ be chosen such that $b_{\Uc}(\cdot,\bar{u}_c) :=  \frac{1}{2}\norm{\cdot - \bar{u}_c}^2$ for a fixed $\bar{u}_c\in\dom{g}$, and $\gamma_1 > 0$.
Then, for $k\geq 1$, if $\tau_k, \gamma_k, \beta_k$, and $\eta_k$  are updated by
\begin{equation}\label{eq:update_rules_ama}
\tau_k := \frac{3}{k+4},~\gamma_k := \frac{5\gamma_1}{k+4}, ~\beta_k := \frac{18\norm{\Ab}^2(k+5)}{5\gamma_1(k+1)(k+7)},~\text{and}~~\eta_k := \frac{5\gamma_1}{2\norm{\Ab}^2(k+5)},
\end{equation}
then they satisfy conditions \eqref{eq:ama_param_cond}.
Moreover, the convergence rate of $\{\tau_k\}$ is optimal, and $\beta_k \leq \frac{18\norm{\Ab}^2}{5\gamma_1(k+1)}$.
\end{lemma} 

Let us comment here on our weighting strategy and its relation to \cite{Davis2014b}, which places emphasis on the later iterates in averaging by using $\omega_i = i+1$ as described by \eqref{eq: weighting} in Section 7. In our updates, we consider another weighting scheme \eqref{eq: weighting} that places even more emphasis. For this purpose, we use $\omega_i = (i+1)(i+2)$ and rewrite \eqref{eq: weighting} in a way to mimic the averaging step in \eqref{eq:add_ama_steps}: $ \bar{\xb}^{k\!+\!1} = \frac{1}{k+4} \bar{\xb}^{k}  +  \frac{3}{k+4} \xb^{k\!+\!1}.$ Hence,  our particular primal weighting scheme \eqref{eq:new_ama_alg} uses $\tau_k = \frac{3}{k+4}$. 

\beforesubsec
\subsection{The new smoothing AMA algorithm}
\aftersubsec
Since $\lbd^{\ast}_k$ in the first line of \eqref{eq:add_ama_steps} requires one matrix-vector multiplication $(\Ab\ub, \Bb\vb)$, we can combine the third line of \ref{eq:new_ama_alg} and the second line of \eqref{eq:add_ama_steps} to compute $\lbd^{\ast}_k$ recursively as
\vspace{-0.75ex}
\begin{equation}\label{eq:lbd_star_k}  
\lbd^{\ast}_{k\!+\!1} := \beta_{k\!+\!1}^{-1}\big[(1-\tau_k)\beta_k\lbd^{\ast}_k + \tau_k\eta_k^{-1}(\bar{\lbd}^{k\!+\!1} - \hat{\lbd}^k)\big].
\vspace{-0.75ex}
\end{equation}
Consequently, each iteration of Algorithm~\ref{alg:SAMA} below requires one matrix-vector multiplication $(\Ab\ub, \Bb\vb)$ and one corresponding adjoint operation $(\Ab^{\top}\lbd,\Bb^{\top}\lbd)$. 
Hence, the per-iteration complexity of \eqref{eq:new_ama_alg} and the standard AMA \eqref{eq:ama_scheme} are essentially the same. 
Finally, we can combine the main steps \eqref{eq:new_ama_alg}, \eqref{eq:add_ama_steps}, \eqref{eq:lbd_star_k}, and the update rule \eqref{eq:update_rules_ama} to complete the smoothing alternating minimization algorithm (\ref{eq:new_ama_alg})  in Algorithm \ref{alg:SAMA}.

\begin{algorithm}[ht!]\caption{(\textit{Smoothing Alternating Minimization Algorithm $($\ref{eq:new_ama_alg}$)$})}\label{alg:SAMA}
\begin{normalsize}
\begin{algorithmic}[1]
\Statex{\hskip-4ex}\textbf{Initialization:} 
\State\label{step:sama_init1}Fix $ \bar{\ub}_c\in\dom{g}$. 
Choose $\hat{\lbd}^0 \in\R^n$ and $\gamma_1 > 0$. Set  $\eta_0 := \frac{\gamma_1}{2\norm{\Ab}^2}$ and $\beta_1 := \frac{27\norm{\Ab}^2}{20\gamma_1}$.
\State\label{step:sama_init2a}Compute $\bar{\ub}^1  := \prox_{\gamma_1^{-1}g}\big( \bar{\ub}_c + \gamma_1^{-1}\Ab^{\top}\hat{\lbd}^0\big)$.
\State\label{step:sama_init2b}Solve $\bar{\vb}^1  := \mathrm{arg}\displaystyle\min_{\vb} \big\{ h(\vb)  -  \iprods{\hat{\lbd}^0, \Bb\vb}  + \frac{\eta_0}{2}\Vert\Ab\bar{\ub}^1 + \Bb\vb - \cb\Vert^2\big\}$.
\State\label{step:sama_init2c}Update $\bar{\lambda}^1 := \hat{\lambda}^0 - \eta_0(\Ab\bar{\ub}^1 + \Bb\bar{\vb}^1 - \cb)$ and $\lbd^{\ast}_1 := \beta_1^{-1}(\cb - \Ab\bar{\ub}^1 - \Bb\bar{\vb}^1)$.
\vspace{1ex}
\Statex{\hskip-4ex}\textbf{Iteration:} \textbf{For}~{$k=1$ {\bfseries to} $k_{\max}$, \textbf{perform:}}
\vspace{0.5ex}
\State\label{step:sama_update_pars} Compute $\tau_k := \frac{3}{k+4}$, $\gamma_{k\!+\!1} := \frac{5\gamma_1}{k+5}$, $\beta_k := \frac{18\norm{\Ab}^2(k+5)}{5\gamma_1(k+1)(k+7)}$ and $\eta_{k} := \frac{5\gamma_1}{2\norm{\Ab}^2(k+5)}$.
\vspace{0.5ex}
\State\label{step:sama_update_lbdh} Set $\hat{\lbd}^k := (1-\tau_k)\bar{\lbd}^k + \tau_k\lbd^{\ast}_k$.
\vspace{0.5ex}
\State\label{step:sama_update_uhat} Compute $\hat{\ub}^{k\!+\!1}   := \prox_{\gamma_{k\!+\!1}^{-1}g}\big( \bar{\ub}_c + \gamma_{k\!+\!1}^{-1}\Ab^{\top}\hat{\lbd}^k\big)$.
\State\label{step:sama_update_vhat} Solve $\hat{\vb}^{k\!+\!1}  := \mathrm{arg}\displaystyle\min_{\vb} \big\{ h(\vb)  -  \iprods{\hat{\lbd}^k, \Bb\vb}  + \frac{\eta_k}{2}\Vert\Ab\hat{\ub}^{k\!+\!1} + \Bb\vb - \cb\Vert^2\big\}$.
\State\label{step:sama_update_lbdb} Update $\bar{\lbd}^{k\!+\!1} := \hat{\lbd}^k - \eta_k(\Ab\hat{\ub}^{k\!+\!1} + \Bb\hat{\vb}^{k\!+\!1} - \cb)$.
\vspace{0.5ex}
\State\label{step:sama_update_lbds} Compute $\lbd^{\ast}_{k\!+\!1} := \beta_{k\!+\!1}^{-1}\big[(1-\tau_k)\beta_k\lbd^{\ast}_k + \tau_k\eta_k^{-1}(\bar{\lbd}^{k\!+\!1} - \hat{\lbd}^k)\big]$.
\vspace{0.5ex}
\State\label{step:sama_update_uvbar} Update $\bar{\ub}^{k\!+\!1} := (1-\tau_k)\bar{\ub}^k  + \tau_k\hat{\ub}^{k\!+\!1}$ and $\bar{\vb}^{k\!+\!1} := (1-\tau_k)\bar{\vb}^k +  \tau_k\hat{\vb}^{k\!+\!1}$.
\Statex{\hskip-4ex}\textbf{End~for}
\end{algorithmic}
\end{normalsize}
\end{algorithm}

We can view Algorithm~\ref{alg:SAMA} as a primal-dual method, where we apply Nesterov's accelerated method to the smoothed dual problem while using a weighted averaging scheme $\bar{x}^{k} = \big(\sum_{i=0}^k\omega_i\big)^{-1}\sum_{i=0}^k\omega_i\hat{\xb}^i$ for the primal variables.
However, Algorithm \ref{alg:SAMA} aims at solving the nonsmooth problem \eqref{eq:constr_cvx} without any additional assumption on $g$ and $h$ except for the finiteness of $D_f$ in \eqref{eq:Df_diameter}.

\beforesubsec
\subsection{Convergence analysis}
\aftersubsec
We prove in Appendix \ref{apdx:th:new_ama} the convergence and the worst-case iteration-complexity of Algorithm~\ref{alg:SAMA} in Theorem~\ref{th:new_ama}.

\begin{theorem}\label{th:new_ama}
Assume that $b_{\Uc}$ is chosen as $b_{\Uc}(\cdot,\bar{u}_c) := \frac{1}{2}\norm{\cdot - \bar{u}_c}^2$ for any fixed $\bar{u}_c\in\dom{g}$.
Let $\{\bar{\wb}^k\}$ be the sequence generated by Algorithm \ref{alg:SAMA}. Then, for any $\gamma_1 > 0$, the following estimates hold
\begin{equation}\label{eq:convergence_AMA}
{\!\!}\left\{  \begin{array}{ll}
f(\bar{\xb}^k) - \fopt &\leq \frac{5\gamma_1}{(k+4)}\left(\frac{\norm{\bar{\ub}^c -\ub^{\star}}^2}{2} + \frac{9D_f^2}{8\norm{\Ab}^2(k+3)}\right), \vspace{1ex}\\
\norm{\Ab\bar{\ub}^k \!+\! \Bb\bar{\vb}^k \!-\! \cb} &\leq \frac{36\norm{\Ab}^2\Vert\lbd^{\star}\Vert}{5\gamma_1(k+1)} + \frac{6\norm{\Ab}}{(k+1)}\sqrt{\frac{\norm{\bar{\ub}^c -\ub^{\star}}^2}{2} + \frac{9D_f^2}{8\norm{\Ab}^2(k+7)}},\vspace{1ex}\\
d(\bar{\lbd}^k) - d^{\star} &\leq \frac{36\norm{\Ab}^2\Vert\lbd^{\star}\Vert^2}{5\gamma_1(k\!+\!1)} \!+\! \frac{6\norm{\Ab}\Vert\lbd^{\star}\Vert}{(k\!+\!1)}\sqrt{\frac{\norm{\bar{\ub}^c -\ub^{\star}}^2}{2} \!+\! \frac{9D_f^2}{8\norm{\Ab}^2(k+7)}} \vspace{0.75ex}\\
& \!+\! \frac{5\gamma_1}{(k+4)}\left(\frac{\norm{\bar{\ub}^c -\ub^{\star}}^2}{2} \!+\! \frac{9D_f^2}{8\norm{\Ab}^2(k+3)}\right),
\end{array}\right.{\!\!}
\end{equation}
where $D_f$ are defined by \eqref{eq:Df_diameter}.
As a consequence, if we choose $\gamma_1 := \norm{\Ab}$, then the worst-case iteration-complexity of Algorithm \ref{alg:SAMA} to achieve an $\varepsilon$-primal solution $\bar{\xb}^k$ of \eqref{eq:constr_cvx} in the sense of Definition~\ref{de:approx_sol} is $\mathcal{O}\left(\varepsilon^{-1}\right)$.
\end{theorem}

Theorem \ref{th:new_ama} shows that the convergence rate of Algorithm \ref{alg:SAMA} consists of two parts. 
While the first part depends on $\Vert\bar{\ub}^c - \ub^{\star}\Vert^2$ which is only $\mathcal{O}(1/k)$, the second part depending on $D_f$ is up to $\mathcal{O}(1/k^2)$.
We can obtain the convergence rate of the feasibility gap $\norm{\Ab\bar{\ub}^k + \Bb\bar{\vb}^k - \cb}$ from the dual convergence as done in \cite{Goldstein2012}. However, this rate is only $\mathcal{O}(1/\sqrt{k})$ when the rate on the dual $d(\bar{\lbd}^k) - d^{\star}$ is  $\mathcal{O}(1/k)$.

\begin{remark}\label{re:artificial_bound}
If Assumption~\ref{as:A2} fails to hold, then artificial constraints $\norm{u} \leq R$ and/or $\norm{v}\leq R$ must be added to \eqref{eq:constr_cvx}.
Since Algorithm~\ref{alg:SAMA} does not require $R$ as an input, we can estimate $R$ after we terminate this algorithm.
Theoretically, the sequence $\set{(\bar{u}^k, \bar{v}^k)}$ generated by Algorithm~\ref{alg:SAMA} converges to $x^{\star} = (u^{\star},v^{\star})$ a solution of \eqref{eq:constr_cvx}.
Hence, by Lemma~\ref{le:restricted_domain}, $R$ can roughly be estimated as $R > \sup_k\set{\norm{\bar{u}^k}, \norm{\bar{v}^k}}$.
Note that, in this case, the objective function of the subproblems in $u$ and $v$ from \eqref{eq:new_ama_alg} is also changed from $g$ to $g + \delta_{\Ball_R}$, and from $h$ to $h+\delta_{\Ball_R}$, respectively.
Practically, by assuming that $R$ is sufficiently large so that  $\norm{u} \leq R$ and $\norm{v}\leq R$ are inactive, we can discard the term $\delta_{\Ball_R}(u)$, and $\delta_{\Ball_R}(v)$. 
Therefore, the computation of $\hat{u}^{k+1}$ and $\hat{v}^{k+1}$ at Step \ref{step:sama_update_uhat} and Step \ref{step:sama_update_vhat}, respectively, of Algorithm~\ref{alg:SAMA} is unchanged.
\end{remark}

\beforesubsec
\subsection{Special case: $g$ is strongly convex}\label{subsec:special_cases}
\aftersubsec
We now consider a special case of the constrained problem \eqref{eq:constr_cvx} when $g$ is strongly convex.
If $g$ is strongly convex with the convexity parameter $\mu_{g} > 0$, then we can modify Algorithm \ref{alg:SAMA} so that $d(\bar{\lambda}^k) - d^{\star} \leq \mathcal{O}(\frac{1}{k^2})$ in terms of the dual objective function as shown in \cite{Goldstein2012}. 
However, the convergence rate in terms of the primal objective residual $f(\bar{\xb}^k) - \fopt$ and the primal feasibility gap $\norm{\Ab\bar{\ub}^k + \Bb\bar{\vb}^k - \cb}$ we can prove is worse than $\mathcal{O}(\frac{1}{k^2})$. 

Let us consider again the dual function $\varphi$ defined by \eqref{eq:dual_func_i}. Since $g$ is strongly convex with the strong convexity parameter $\mu_g > 0$, $\nabla{\varphi}$ is Lipschitz continuous with the Lipschitz constant $L_{\varphi} := \frac{\norm{\Ab}^2}{\mu_g}$.
We modify Algorithm \ref{alg:SAMA} in order to obtain a new variant that captures the strong convexity of $g$ and removes the smoothness parameter $\gamma_k$.
By a similar analysis as in Lemma \ref{le:ama_key_estimate2}, we can show in Appendix \ref{apdx:co:ama_variant2} that if the following conditions  hold
\begin{align}\label{eq:ama_param_cond_scvx}
\beta_{k\!+\!1} \geq (1-\tau_k)\beta_k ~~\text{and}~~\eta_k\Big(\frac{3}{2} +  \tau_k - \frac{\norm{\Ab}^2\eta_k}{\mu_g}\Big) \geq \frac{\tau_k^2}{(1-\tau_k)\beta_k},
\end{align}
then
\begin{equation}\label{eq:ama_key_estimate_scvx}
G_{\beta_{k\!+\!1}}(\bar{\wb}^{k\!+\!1}) \leq (1-\tau_k)G_{\beta_k}(\bar{\wb}^k) +  \frac{\tau^2_k\eta_kD_f^2}{4},
\end{equation}
where $G_{\beta_k}(\bar{\wb}^k) := f_{\beta_k}(\bar{\xb}^k) + d(\bar{\lbd}^k)$.
The first iterate $\bar{\ub}^1$ in \eqref{eq:init_point} can be  computed as
\begin{equation}\label{eq:ama_init_point_scvx}
\bar{\ub}^1  := \mathrm{arg}\!\!\!\!\!\!\min_{\ub \in\dom{f}}\big\{ g(\ub) + \iprods{\hat{\lbd}^0, \Ab\ub} \big\}.
\end{equation}
Using \eqref{eq:ama_init_point_scvx} and new update rules for the parameters in Algorithm \ref{alg:SAMA}, we obtain a new variant of Algorithm \ref{alg:SAMA}.
The following corollary shows the convergence of this variant, whose proof is also moved to Appendix \ref{apdx:co:ama_variant2}.

\begin{corollary}\label{co:ama_variant2}
Let $\{\bar{\wb}^k\}$ be the sequence generated by Algorithm \ref{alg:SAMA} using \eqref{eq:ama_init_point_scvx} and the update rules
\begin{align}\label{eq:update_params_ama_scvx}
{\!\!}\tau_k := \frac{3}{k+4}, ~~\eta_k := \frac{\mu_g}{2\norm{\Ab}^2}, ~~\text{and}~~\beta_k := \frac{2\norm{\Ab}^2\tau_k^2}{\mu_g(1-\tau_k^2)} = \frac{18\norm{\Ab}^2}{\mu_g(k+1)(k+7)}.{\!\!\!}
\end{align}
Then, the following estimates hold
\begin{equation}\label{eq:convergence_AMA_scvx}
\left\{\begin{array}{lll}
f(\bar{\xb}^k) - \fopt &\leq \frac{9\mu_gD_f^2}{16\norm{\Ab}^2(k+3)} & =\mathcal{O}\left(\frac{1}{k}\right),\vspace{1.5ex}\\
\norm{\Ab\bar{\ub}^k + \Bb\bar{\vb}^k - \cb} &\leq \frac{36\norm{\Ab}^2\Vert\lbd^{\star}\Vert}{\mu_g(k+1)(k+7)} + \frac{9D_f}{2\sqrt{(k+1)(k+3)(k+7)}} & =\mathcal{O}\left(\frac{1}{k^{3/2}}\right).
\end{array}\right.
\end{equation}
Alternatively, if we use the following update rules in Algorithm \ref{alg:SAMA}
\begin{align}\label{eq:update_params_ama_scvx2}
{\!\!}\tau_k := \frac{3}{k+4}, ~~\eta_k := \frac{\mu_g\tau_k}{\norm{\Ab}^2}, ~~\text{and}~~\beta_k := \frac{2\norm{\Ab}^2\tau_k}{3\mu_g(1-\tau_k)} = \frac{2\norm{\Ab}^2}{\mu_g(k+1)},{\!\!\!}
\end{align}
then
\begin{equation}\label{eq:convergence_AMA_scvx2}
\left\{\begin{array}{lll}
f(\bar{\xb}^k) - \fopt &\leq \frac{27\mu_gD_f^2}{4\norm{\Ab}^2(k+3)^2} & =\mathcal{O}\left(\frac{1}{k^2}\right),\vspace{1.5ex}\\
\norm{\Ab\bar{\ub}^k + \Bb\bar{\vb}^k - \cb} &\leq \frac{4\norm{\Ab}^2\Vert\lbd^{\star}\Vert}{\mu_g(k+1)} + \frac{3\sqrt{3}}{(k+3)}\frac{D_f}{\sqrt{k\!+\!1}} & =\mathcal{O}\left(\frac{1}{k}\right).
\end{array}\right.
\end{equation}
Here, $D_f$ is defined by \eqref{eq:Df_diameter}.
In both cases, the guarantee of the primal-dual gap function $G(\wbar^k) := f(\bar{\xb}^k) + d(\bar{\yb}^k)$ is 
\begin{equation}\label{eq:primal_dual_guarantee}
G(\wbar^k) + \frac{1}{2\beta_k}\norm{\Ab\bar{\ub}^k + \Bb\bar{\vb}^k - \cb}^2 \leq \frac{9\mu_gD_f^2}{4\norm{\Ab}^2(k+3)},
\end{equation}
where $\beta_k$ is given by either \eqref{eq:update_params_ama_scvx} or \eqref{eq:update_params_ama_scvx2}.
\end{corollary}

We note that, similar to \cite{Goldstein2012}, if we modify Step~\ref{step:sama_update_lbds} of Algorithm~\ref{alg:SAMA} by $\lbd^{\ast}_{k+1} := \lbd^{\ast}_k + \frac{1}{\tau_k}\big(\bar{\lbd}^{k+1} - \hat{\lbd}^k\big)$, then we can prove the $\mathcal{O}(\frac{1}{k^2})$-convergence rate for the dual objective residual $d(\lbd^k) - d^{\star}$ in Algorithm \ref{alg:SAMA} under the strong  convexity of $g$.

%

 \vspace{1ex}
 \beforesubsec
\subsection{Composite convex minimization involving a linear operator}
\vspace{-1ex}
\aftersubsec
A common composite convex minimization formulation in image processing and machine learning \cite{Bauschke2011} is the following problem:
\begin{equation}\label{eq:composite_cvx}
\min_{\ub\in\R^p_1} \set{ f(\ub) := g(\ub) + h(F\ub - \yb) },
\end{equation}
where $g$ and $h$ are two proper, closed and convex functions (possibly nonsmooth), $F$ is a linear operator from $\R^{p_1}$ to $\R^n$, and $\yb\in\R^n$ is a given observed vector.
We are more interested in the case that $g$ and $h$ are nonsmooth but are equipped with a tractable proximal operator. For example, $g$ and $h$ are both the $\ell_1$-norm.

Classical AMA and ADMM methods can solve \eqref{eq:composite_cvx} but do not have an $\mathcal{O}(1/k)$ - theoretical convergence rate guarantee without additional smoothness-type or strong convexity-type assumption on $g$ and $h$.
In addition, the ADMM still requires to solve the subproblem at the second line of \eqref{eq:admm_scheme} iteratively when $F$ is not orthogonal.

If we introduce a new variable $\vb := F\ub - \yb$, then we can reformulate \eqref{eq:composite_cvx} into \eqref{eq:constr_cvx} with $\Ab = F$ and $\Bb = -\mathbb{I}$.
In this case, we can apply both Algorithm~\ref{alg:SAMA} and Algorithm~\ref{alg:SADMM} (in Section~\ref{sec:new_admm}) to solve the resulting problem without additional assumption on $g$ and $h$ except for the boundedness of $D_f$.
However, we only focus on Algorithm~\ref{alg:SAMA}, which only requires the proximal operator of $g$ and $h$.
The main step of this algorithmic variant can be written explicitly as
\begin{equation*}
\left\{\begin{array}{ll}
\hat{\ub}^{k\!+\!1} &:= \prox_{\gamma_{k\!+\!1}^{-1}g}\big( \bar{\ub}_c + \gamma_{k\!+\!1}^{-1}F^{\top}\hat{\lbd}^k\big),\vspace{1ex}\\
\hat{\vb}^{k\!+\!1} &:= \prox_{\eta_k^{-1}h}\big(F\hat{\ub}^{k\!+\!1} - y - \eta_k^{-1}\hat{\lbd}^k\big).
\end{array}\right.
\end{equation*}
Substituting this step into Algorithm~\ref{alg:SAMA}, we obtain a new variant for solving \eqref{eq:composite_cvx} using only the proximal operator of $g$ and $h$.

\beforesec
\section{The new smoothing ADMM method}\label{sec:new_admm}
\aftersec
For completeness, we present a new alternating direction method of multipliers (ADMM) algorithm for solving \eqref{eq:constr_cvx} by applying Douglas-Rachford splitting method to the smoothed dual problem.
Our new algorithm, dubbed the smoothing ADMM \eqref{eq:new_admm}, features  similar optimal convergence rate guarantees as  \ref{eq:new_ama_alg}. See Section 7 for further discussion. 

\beforesubsec
\subsection{The main steps of the smoothing ADMM method}
\aftersubsec
The main step of our \ref{eq:new_admm} scheme  is as follows.
Given $\hat{\lbd}^k\in\R^n$, $\hat{\vb}^k \in \dom{h}$ and the parameters  $\gamma_{k\!+\!1} > 0$, $\rho_k > 0$ and $\eta_k > 0$, we compute $(\hat{\ub}^{k\!+\!1}, \hat{\vb}^{k\!+\!1}, \bar{\lbd}^{k\!+\!1})$ as follows:
\begin{equation}\label{eq:new_admm}
\!\!\!\!\left\{ \begin{array}{ll}
\hat{\ub}^{k\!+\!1} &  := \displaystyle\mathrm{arg}{\!\!\!\!\!}\min_{{\!\!\!}\ub\in\dom{g}}\!\!\Big\{ g(\ub; \gamma_{k+1}) \!-\! \iprods{\Ab^{\top}\hat{\lbd}^k, \ub} \!+\! \frac{\rho_k}{2}\norm{\Ab\ub \!+\! \Bb\hat{\vb}^k \!-\! \cb}^2 \Big\},\vspace{0.75ex}\\
\hat{\vb}^{k\!+\!1} & := \displaystyle\mathrm{arg}{\!\!\!\!\!}\min_{\vb\in\dom{h}}\!\Big\{ h(\vb) - \iprods{\Bb^{\top}\hat{\lbd}^k, \vb} + \frac{\eta_k}{2}\norm{\Ab\hat{\ub}^{k\!+\!1} + \Bb\vb - \cb}^2\! \Big\},\vspace{0.75ex}\\
\bar{\lbd}^{k\!+\!1} & := \hat{\lbd}^k -  \eta_k\big(\Ab\hat{\ub}^{k\!+\!1} + \Bb\hat{\vb}^{k\!+\!1} - \cb\big),
\end{array}\right.{\!\!\!\!\!\!\!}
\tag{$\mathrm{SADMM}$}
\end{equation}
where $g(u; \gamma) := g(u) + \gamma b_{\Uc}(\ub,\bar{\ub}^c)$.
This scheme is different from the standard ADMM scheme \eqref{eq:admm_scheme} at two points. 
First, $\hat{\ub}^{k\!+\!1}$ is computed from  the regularized subproblem with $g(\cdot; \gamma)$ instead of $g$. 
Second, we use different penalty parameters $\rho_k$ and $\eta_k$ compared to the standard ADMM scheme \eqref{eq:admm_scheme} in Section~\ref{sec:conclusion}. 
The complexity of computing $\hat{\ub}^{k+1}$ in \eqref{eq:new_admm} is essentially the same as computing $\ub^{k+1}$ in the standard ADMM scheme~\eqref{eq:admm_scheme} below.

As a special case, if $\Ab = \mathbb{I}$, the identity operator,  or $\Ab$ is orthonormal, then we can choose $b_{\Uc}(\cdot,\bar{\ub}^c) = (1/2)\Vert\cdot-\bar{\ub}^c\Vert^2$ to obtain a closed form solution of $\hat{\ub}^{k+1}$ as
\begin{equation*}
\begin{array}{ll}
\hat{\ub}^{k+1} &:= \prox_{(\rho_k+\gamma_{k+1})^{-1}g}\left((\rho_k + \gamma_{k+1})^{-1}\big(\gamma_{k+1}\bar{\ub}^c +  \Ab^{\top}(\hat{\lbd}^k - \rho_k(\Bb\hat{\vb}^k - \cb))\big)\right).
\end{array}
\end{equation*}
In addition to \eqref{eq:new_admm}, our algorithm also requires additional steps
\begin{equation}\label{eq:add_admm_steps}
{\!\!\!}\left\{\begin{array}{llll}
& \hat{\lbd}_k &:= (1-\tau_k)\bar{\lbd}^k + \tau_k\lbd^{\ast}_k, & \text{(dual acceleration)}\vspace{0.75ex}\\
& (\bar{\ub}^{k\!+\!1}, \bar{\vb}^{k\!+\!1}) &:= (1-\tau_k)(\bar{\ub}^{k}, \bar{\vb}^{k}) + \tau_k(\hat{\ub}^{k\!+\!1}, \hat{\vb}^{k\!+\!1}), &\text{(weighted averaging)}
\end{array}\right.{\!\!\!}
\end{equation} 
 as in Algorithm \ref{alg:SAMA}, where $\lbd^{\ast}_k := \beta^{-1}_k(\cb - \Ab\bar{\ub}^k - \Bb\bar{\vb}^k)$, and $\tau_k \in (0, 1)$ is a  step size.

We prove in Appendix \ref{apdx:le:key_estimate2_admm}  the following lemma, which provides  conditions on the parameters to guarantee the gap reduction condition.

\begin{lemma}\label{le:key_estimate2_admm}
Let $\{\bar{\wb}^k \}$ with $\bar{\wb}^k := (\bar{\ub}^k, \bar{\vb}^k, \bar{\lbd}^k)$ be the sequence generated by \eqref{eq:new_admm}-\eqref{eq:add_admm_steps}. 
If $\tau_k \in (0, 1)$ and $\gamma_k, \beta_k, \rho_k, \eta_k \in \R_{++}$ satisfy
\begin{equation}\label{eq:param_conds}
\left\{\begin{array}{lllll}
(1-\tau_k)(1+2\tau_k)\eta_k\beta_k &\geq 2\tau_k^2,  & &\gamma_{k\!+\!1} &\geq \left(\frac{3 - 2\tau_k}{3 - (2-L_b^{-1})\tau_k}\right)\gamma_k,\vspace{1ex}\\
\beta_{k\!+\!1} &\geq (1-\tau_k)\beta_k, &\text{and}&~\gamma_{k\!+\!1} &\geq \Vert\Ab\Vert^2\left(\eta_k + \frac{\rho_k}{\tau_k}\right),
\end{array}\right.
\end{equation}
then the following non-monotone gap reduction condition holds
\begin{equation}\label{eq:key_estimate0}
G_{\gamma_{k\!+\!1}\beta_{k\!+\!1}}(\bar{\wb}^{k\!+\!1}) \leq (1-\tau_k)G_{\gamma_k\beta_k}(\bar{\wb}^k) + \left(\frac{\tau_k^2\eta_k}{4} + \frac{\tau_k\rho_k}{2}\right)D_f^2,
\end{equation}
where $G_{\gamma_k\beta_k}$ is defined by \eqref{eq:G_gammabeta}, and $D_f$ is defined by \eqref{eq:Df_diameter}.
\end{lemma}

\beforesubsec
\subsection{Updating parameters}
\aftersubsec
The second step of our algorithmic design is to derive an update rule for the parameters to satisfy the conditions \eqref{eq:param_conds}.
Lemma~\ref{le:update_admm_params} shows \textbf{one possibility} to update these parameters, whose proof is given in Appendix~\ref{apdx:le:update_admm_params}.  

\begin{lemma}\label{le:update_admm_params}
Let $b_{\Uc}$ be chosen such that $b_{\Uc}(\cdot,\bar{u}_c) :=  \frac{1}{2}\norm{\cdot - \bar{u}_c}^2$ for a fixed $\bar{u}_c\in\dom{g}$, and $\gamma_1 > 0$.
Then, for $k\geq 1$,  $\tau_k$, $\gamma_k$, $\beta_k$, $\rho_k$, and $\eta_k$ updated by
\begin{equation}\label{eq:update_admm_params}
\begin{array}{llll}
&\tau_k  := \frac{3}{k + 4}, &\gamma_{k} := \frac{3\gamma_1}{k + 2}, &\beta_k := \frac{6\norm{\Ab}^2(k+3)}{\gamma_1(k + 1)(k + 10)},\vspace{1ex}\\
&\rho_k  := \frac{9\gamma_1}{2\norm{\Ab}^2(k+3)(k + 4)}, &\eta_k := \frac{3\gamma_1}{2\norm{\Ab}^2(k + 3)}, &
\end{array}
\end{equation}
satisfy  \eqref{eq:param_conds}.
Moreover, $\beta_k \leq \frac{9\norm{\Ab}^2}{5\gamma_1(k+1)}$, and the convergence rate of $\{\tau_k\}$ is optimal.
\end{lemma}

We note that we have freedom to choose $\gamma_1$ in order to trade off the primal objective residual $f(\bar{\xb}^k) - \fopt$ and the primal feasibility gap $\norm{\Ab\bar{\ub}^k + \Bb\bar{\vb}^k - \cb}$ as in Algorithm \ref{alg:SAMA}.

\beforesubsec
\subsection{The smoothing ADMM algorithm}
\aftersubsec
Similar to Algorithm \ref{alg:SAMA}, we can  combine the third line of \eqref{eq:new_admm} and the second line of \eqref{eq:add_admm_steps} to update $\lbd^{\ast}_k$.
In this case, the arithmetic cost-per-iteration of Algorithm \ref{alg:SADMM} is essentially the same as in the standard ADMM scheme \eqref{eq:admm_scheme}.
We also use $\bar{\wb}^1 = (\bar{\ub}^1, \bar{\vb}^1, \bar{\lbd}^1)$ computed by \eqref{eq:init_point} at the first iteration.
By putting \eqref{eq:init_point}, \eqref{eq:update_admm_params}, \eqref{eq:new_admm}, \eqref{eq:add_admm_steps} and \eqref{eq:lbd_star_k}  together, we obtain a complete \ref{eq:new_admm} algorithm  as presented in Algorithm \ref{alg:SADMM}.

\begin{algorithm}[ht!]\caption{(\textit{Smoothing Alternating Direction Method-of-Multipliers $(\mathrm{SAD\text{-}MM})$})}\label{alg:SADMM}
\begin{normalsize}
\begin{algorithmic}[1]
\Statex{\hskip-4ex}\textbf{Initialization:} 
\State\label{step:sadmm_init1}Fix $ \bar{\ub}_c\in\dom{g}$. 
Choose $\hat{\lbd}^0 \in\R^n$ and $\gamma_1 > 0$. Set  $\eta_0 := \frac{\gamma_1}{2\norm{\Ab}^2}$ and $\beta_1 := \frac{12\norm{\Ab}^2}{11\gamma_1}$.
\State\label{step:sadmm_init2a}Compute $\bar{\ub}^1  := \prox_{\gamma_1^{-1}g}\big( \bar{\ub}_c + \gamma_1^{-1}\Ab^{\top}\hat{\lbd}^0\big)$.
\State\label{step:sadmm_init2b}Solve $\bar{\vb}^1  := \mathrm{arg}\displaystyle\min_{\vb} \big\{ h(\vb)  -  \iprods{\hat{\lbd}^0, \Bb\vb}  + \frac{\eta_0}{2}\Vert\Ab\bar{\ub}^1 + \Bb\vb - \cb\Vert^2\big\}$. Set $\hat{\vb}^1 := \bar{\vb}^1$.
\vspace{0.5ex}
\State\label{step:sadmm_init2c}Update $\bar{\lambda}^1 := \hat{\lambda}^0 - \eta_0(\Ab\bar{\ub}^1 + \Bb\bar{\vb}^1 - \cb)$ and $\lbd^{\ast}_1 := \beta_1^{-1}(\cb - \Ab\bar{\ub}^1 - \Bb\bar{\vb}^1)$.
\vspace{1ex}
\Statex{\hskip-4ex}\textbf{Iteration:} \textbf{For}~{$k=1$ {\bfseries to} $k_{\max}$, \textbf{perform:}}
\vspace{0.5ex}
\State\label{step:sadmm_update_pars} Compute $\tau_k := \frac{3}{k+4}$, $\gamma_{k\!+\!1} := \frac{3\gamma_1}{k+3}$, $\beta_k := \frac{6\norm{\Ab}^2(k+3)}{\gamma_1(k+1)(k+10)}$.
Then, set $\eta_{k} := \frac{3\gamma_1}{2\norm{\Ab}^2(k+3)}$ and $\rho_{k} := \frac{9\gamma_1}{2\norm{\Ab}^2(k+3)(k+4)}$.
\vspace{1ex}
\State\label{step:sadmm_update_lbdh} Set $\hat{\lbd}^k := (1-\tau_k)\bar{\lbd}^k + \tau_k\lbd^{\ast}_k$.
\vspace{1ex}
\State\label{step:sadmm_update_uhat} Solve $\hat{\ub}^{k\!+\!1}   :=  \displaystyle\mathrm{arg}\min_{\ub}\!\big\{ g(\ub) - \iprods{\hat{\lbd}^k, \Ab\ub} \!+\! \frac{\rho_k}{2}\norm{\Ab\ub \!+\! \Bb\hat{\vb}^k \!-\! \cb}^2 + \gamma_{k\!+\!1}b_{\Uc}(\ub,\bar{\ub}^c)\big\}$.
\vspace{1ex}
\State\label{step:sadmm_update_vhat} Solve $\hat{\vb}^{k\!+\!1}  := \mathrm{arg}\displaystyle\min_{\vb} \big\{ h(\vb)  -  \iprods{\hat{\lbd}^k, \Bb\vb}  + \frac{\eta_k}{2}\Vert\Ab\hat{\ub}^{k\!+\!1} + \Bb\vb - \cb\Vert^2\big\}$.
\vspace{1ex}
\State\label{step:sadmm_update_lbdb} Update $\bar{\lbd}^{k\!+\!1} := \hat{\lbd}^k - \eta_k(\Ab\hat{\ub}^{k\!+\!1} + \Bb\hat{\vb}^{k\!+\!1} - \cb)$.
\vspace{1ex}
\State\label{step:sadmm_update_lbds} Compute $\lbd^{\ast}_{k\!+\!1} := \beta_{k\!+\!1}^{-1}\big[(1-\tau_k)\beta_k\lbd^{\ast}_k + \tau_k\eta_k^{-1}(\bar{\lbd}^{k\!+\!1} - \hat{\lbd}^k)\big]$.
\vspace{1ex}
\State\label{step:sadmm_update_uvbar} Update $\bar{\ub}^{k\!+\!1} := (1-\tau_k)\bar{\ub}^k  + \tau_k\hat{\ub}^{k\!+\!1}$ and $\bar{\vb}^{k\!+\!1} := (1-\tau_k)\bar{\vb}^k +  \tau_k\hat{\vb}^{k\!+\!1}$.
\Statex{\hskip-4ex}\textbf{End~for}
\end{algorithmic}
\end{normalsize}
\end{algorithm}

\beforesubsec
\subsection{Convergence analysis}
\aftersubsec
The following theorem with its  proof being in Appendix~\ref{apdx:th:convergence_of_ADMM} shows the worst-case iteration-complexity of Algorithm \ref{alg:SADMM}.

\begin{theorem}\label{th:convergence_of_ADMM}
Assume that $b_{\Uc}$ is chosen as $b_{\Uc}(\cdot,\bar{u}_c) :=  \frac{1}{2}\norm{\cdot - \bar{u}_c}^2$ for a fixed $\bar{u}_c\in\dom{g}$.
Let $\{(\bar{\ub}^k, \bar{\vb}^k, \bar{\lbd}^k)\}$ be the sequence generated by Algorithm \ref{alg:SADMM}. Then the following estimates hold
\begin{equation}\label{eq:convergence_of_admm}
\left\{\begin{array}{ll}
f(\bar{\xb}^k)  - \fopt &  \leq \frac{3\gamma_1}{(k+2)}\left[\frac{\Vert\bar{\ub}^c - \ub^{\star}\Vert^2}{2} + \frac{27D_f^2}{8\norm{\Ab}^2(k\!+\!3)}\right], \vspace{0.75ex}\\
 \Vert\Ab\bar{\ub}^k + \Bb\bar{\vb}^k - \cb\Vert & \leq  \frac{18\norm{\Ab}^2\Vert\lbd^{\star}\Vert}{5\gamma_1(k+1)} + \frac{6\norm{\Ab}}{(k+1)}\sqrt{\Vert\bar{\ub}^c - \ub^{\star}\Vert^2  + \frac{27D_f^2}{8\norm{\Ab}^2(k\!+\!10)}},
\end{array}\right.
\end{equation}
where $D_f$ is given by \eqref{eq:Df_diameter}.
If $\gamma_1 := \norm{\Ab}$, then the worst-case iteration-complexity of Algorithm \ref{alg:SADMM} to achieve an $\varepsilon$ - solution $\bar{\xb}^k$ of \eqref{eq:constr_cvx} is $\mathcal{O}\left(\varepsilon^{-1}\right)$.
\end{theorem}

As can be seen from Theorem \ref{th:convergence_of_ADMM}, the term $\frac{6\norm{\Ab}}{(k+1)}\Big(\frac{\norm{\bar{\ub}^c - \ub^{\star}}^2}{2} + \frac{27D_f^2}{8\norm{\Ab}^2(k+10)}\Big)^{1/2}$ in \eqref{eq:convergence_of_admm} does not depend on the choice of $\gamma_1$. 
If we decrease $\gamma_1$, then the upper bound of $f(\bar{\xb}^k) - \fopt$ decreases, while the upper bound of $\Vert\Ab\bar{\ub}^k + \Bb\bar{\vb}^k  - \cb\Vert$ increases, and vice versa. 
Hence, $\gamma_1$ trades off these worse case bounds.
The convergence rate guarantee on the dual objective residual can be easily obtained from the last bound of \eqref{eq:optimal_bounds}.

\beforesubsec
\subsection{SAMA vs. SADMM}
\aftersubsec
There are at least two cases, where SAMA theoretically gains advantages over SADMM. First, if $\Ab$ is non-orthogonal.  
The $u$-subproblem in \eqref{eq:new_ama_alg} can be computed by using $\prox_{g}$, while in SADMM, the nonorthogonal operator $\Ab$ prevents us from using $\prox_g$.
Second, if $g$ is block separable, i.e., $g(\ub) := \sum_{i=1}^sg_i(\ub_i)$, then we can choose $g(\ub;\gamma) := \sum_{i=1}^s\left[g_i(\ub_i) + \frac{\gamma}{2}\Vert\ub_i - \bar{\ub}_i^c\Vert^2\right]$, which can be evaluated in parallel. 
This is not preserved in SADMM.
Indeed, for SADMM, the subproblem in $\ub$ still has the quadratic term $\frac{\rho_k}{2}\norm{\Ab\ub \!+\! \Bb\hat{\vb}^k-\cb}^2$, which makes it nonseparable even if $g$ is separable.

\beforesec
\section{Numerical evidence}\label{sec:num_experiments}
\aftersec
We illustrate a ``geometric invariant'' property of Algorithm~\ref{alg:SAMA} and Algorithm~\ref{alg:SADMM} for solving the distance minimization problem \eqref{eq:distance problem}.
This problem is classical but solving it efficiently remains an interesting research topic. 
Various algorithms have been proposed including Douglas-Rachford (DR) splitting, Dykstra's projection, and Hauzageau's method as mentioned previously \cite{alotaibi2014best,Bauschke2011}.
In this section, we compare our algorithms with these methods.

We consider the following convex feasibility problem with two convex sets:
\begin{equation}\label{eq:feasible_cvx}
\text{Find~$\lambda^{\star}$ such that:}~\lambda^{\star} \in \Cc_1\cap\Cc_2,
\end{equation}
where $\Cc_1$ and $\Cc_2$ are two nonempty, closed and convex sets in $\R^p$.
Problem \eqref{eq:feasible_cvx} may not have solution.
Hence, instead of solving \eqref{eq:feasible_cvx}, we consider a problem of finding the best substitution for a point in the intersection $\Cc_1\cap\Cc_2$ even if it is empty.
Such a problem can be formulated as 
\begin{equation}\label{eq:distance problem}
d^{\star} := \min_{\lbd\in\R^n} \set{ d(\lbd) := d_{\Cc_1}(\lbd) + d_{\Cc_2}(\lbd) }, 
\end{equation}
where $d_{\Cc}$ is the Euclidean distance to the set $\Cc$. 
Unlike \eqref{eq:feasible_cvx}, the optimal value $d^{\ast}$ of \eqref{eq:distance problem} is always finite as long as $\Cc_1$ and $\Cc_2$ are nonempty.
Moreover, $d^{\star} = \mathrm{dist}(\Cc_1,\Cc_2)$, the distance between $\Cc_1$ and $\Cc_2$.
Hence, if  $\Cc_1\cap\Cc_2 \neq\emptyset$, then $d^{\star} = 0$, see, e.g., \cite{burachik2017approach}.

According to \cite{burachik2017approach},  our primal template \eqref{eq:constr_cvx} for \eqref{eq:distance problem} then takes the following form
\begin{equation}\label{eq:feasible_cvx2}
\min_{\ub, \vb}\set{ s_{\Cc_1}(\ub) + s_{\Cc_2}(\vb) ~~\mid~~ \ub + \vb = 0, ~\ub \in \mathbb{B}_1,  ~\vb \in \mathbb{B}_1},
\end{equation}
where $s_{\Cc_i}$  is the support function of $\Cc_i$ for $i=1,2$, and $\mathbb{B}_r := \set{ \wb \mid \norm{\wb} \leq r}$ for $r>0$.

Clearly, \eqref{eq:feasible_cvx2} is fully nonsmooth, since $s_{\Cc_i}$ is convex and nonsmooth for $i=1,2$.
In addition, \eqref{eq:feasible_cvx2} satisfies Assumption~\ref{as:A2}.
Here, we can even increase the constraint radius, currently 1, to a  sufficiently large number such that the constraints $\ub, \vb\in\mathbb{B}_r$ of each subproblems in \eqref{eq:admm_scheme}, \eqref{eq:new_ama_alg} and \eqref{eq:new_admm} are inactive without changing the underlying problem.
In this particular setting, we can choose the center points for $\ub$ and $\vb$ as zero since they actually obtain the optimal solution. 

If we apply ADMM to solve \eqref{eq:feasible_cvx2}, then it can be written explicitly as 
\begin{equation*} 
\left\{\begin{array}{lll}
\ub^{k\!+\!1} &:= \prox_{\rho^{-1}s_{\Cc_1}}(\lbd^k-\vb^k) &= \lbd^k-\vb^k - \rho^{-1}\pi_{\Cc_1}\left(\rho(\lbd^k-\vb^k)\right),\\
\vb^{k\!+\!1} &:= \prox_{\rho^{-1}s_{\Cc_2}}(\lbd^k - \ub^{k\!+\!1}) &= \lbd^k - \ub^{k\!+\!1} - \rho^{-1}\pi_{\Cc_2}\left(\rho(\lbd^k - \ub^{k\!+\!1})\right),\\
\lbd^{k\!+\!1} &:= \lbd^k - (\ub^{k\!+\!1} + \vb^{k\!+\!1}),&
\end{array}\right.{\!\!\!\!}
\end{equation*}
where $\pi_{\Cc_i}$ is the projection onto $\Cc_i$ for $i=1,2$, and $\rho > 0$ is the penalty parameter.
Clearly, multiplying this expression by $\rho$ and using the same notation, we obtain 
\begin{equation}\label{eq:admm_fb}
\left\{\begin{array}{lll}
\ub^{k\!+\!1} &:= \lbd^k-\vb^k - \pi_{\Cc_1}\left(\lbd^k-\vb^k\right),\\
\vb^{k\!+\!1} &:= \lbd^k - \ub^{k\!+\!1} - \pi_{\Cc_2}\left(\lbd^k - \ub^{k\!+\!1}\right),\\
\lbd^{k\!+\!1} &:= \lbd^k - (\ub^{k\!+\!1} + \vb^{k\!+\!1}),&
\end{array}\right.{\!\!\!\!}
\end{equation}
which shows that this scheme is independent of any parameter $\rho$. 
With an elementary transformation, we can write \eqref{eq:admm_fb} as a Douglas-Rachford (DR) splitting scheme
\begin{equation}\label{eq:DR_fp}
\left\{\begin{array}{ll}
\zb^k &:= \zb^{k-1} + \pi_{\Cc_1}\big(2\lbd^k - \zb^{k-1}\big) - \lbd^k,\\
\lbd^{k\!+\!1} &:= \pi_{\Cc_2}(\zb^k).
\end{array}\right.
\end{equation}
To recover $\ub^k$ and $\vb^k$ from $\zb^k$ and $\lbd^k$, we can use $\ub^k :=  \lbd^{k-1} - \zb^k$ and $\vb^k :=  \zb^{k-1} -  \lbd^k$.

Now, if we apply our SAMA to solve \eqref{eq:feasible_cvx2} using $b_{\Uc}(\ub, \bar{\ub}^c) := (1/2)\Vert \ub - \bar{\ub}^c\Vert^2$, the two main steps of SAMA becomes
\begin{equation}\label{eq:SAMA_app2}
\hspace{-3ex}\left\{\begin{array}{lll}
\hat{\ub}^{k\!+\!1} &{\!\!\!\!\!} := \prox_{\gamma_{k\!+\!1}^{-1}s_{\Cc_1}}(\bar{\ub}^c + \gamma_{k\!+\!1}^{-1}\hat{\lbd}^k) &{\!\!\!\!\!} = \gamma_{k\!+\!1}^{-1}\hat{\lbd}^k + \bar{\ub}^c - \gamma_{k\!+\!1}^{-1}\pi_{\Cc_1}\left( \hat{\lbd}^k + \gamma_{k\!+\!1}\bar{\ub}^c\right),\\
\hat{\vb}^{k\!+\!1} &{\!\!\!\!\!} := \prox_{\eta_k^{-1}s_{\Cc_2}}(\eta_k^{-1}\hat{\lbd}^k - \hat{\ub}^{k\!+\!1}) &{\!\!\!\!\!} = \eta_k^{-1}\hat{\lbd}^k - \hat{\ub}^{k\!+\!1} - \eta_k^{-1}\pi_{\Cc_2}\left(\hat{\lbd}^k - \eta_k\hat{\ub}^{k\!+\!1}\right).
\end{array}\right.\hspace{-3ex}
\end{equation}
Clearly, the standard AMA is not applicable to solve  \eqref{eq:feasible_cvx2} due to the lack of strong convexity.
The standard ADMM applying to \eqref{eq:feasible_cvx2} becomes the alternative projection scheme \eqref{eq:DR_fp} for solving \eqref{eq:feasible_cvx}.
This scheme can be arbitrarily slow if the geometry between two sets $\Cc_1$ and $\Cc_2$ is ill-posed.

To observe an interesting convergence behavior, we test  Dykstra's projection, Hauzageau's method, and  the ADMM \eqref{eq:admm_fb} (or its DR form \eqref{eq:DR_fp}), and  compare them with our algorithms in the following configuration.

We first choose  $\Cc_i := \set{ \ub \in\R^n \mid \iprods{\mathbf{a}_i,\ub} \leq b_i}$ for $i=1,2$ as two half-planes in $\R^n$, where $b_1 = b_2 = 0$.
Here, the normal vectors are   $\mathbf{a}_1 := (\epsilon,  \cdots, \epsilon, -1, \cdots, -1)^{\top}$, and $\mathbf{a}_2 := (0, \cdots, 0, 1, \cdots, 1)^{\top}$, where $\epsilon > 0$ is a positive angle.
The tangent angle $\epsilon$ is repeated $\lfloor n/2\rfloor$ times in $\mathbf{a}_1$, and the zero is repeated $\lfloor n/2\rfloor$ times in $\mathbf{a}_2$, where $n = 1000$.
The starting point is chosen as $\ub^0 := (1,\cdots, 1)^{\top}$. 
By varying $\epsilon$, we can observe the convergence behavior of these five methods.

We note that Dykstra's and Hauzageau's algorithms directly solve the dual problem~\eqref{eq:distance problem}, while our methods and ADMM solve both the primal and dual problems~\eqref{eq:feasible_cvx2} and \eqref{eq:distance problem}. 
We compare these algorithms on the absolute dual objective residual $d(\lambda) - d^{\star}$ of \eqref{eq:distance problem}. 

Figure~\ref{fig:feasibility_prob} shows the convergence of five algorithms with different choices of $\epsilon$.
\begin{figure}[ht!]
\vspace{-2ex}
\begin{center}
\includegraphics[width=0.98\textwidth]{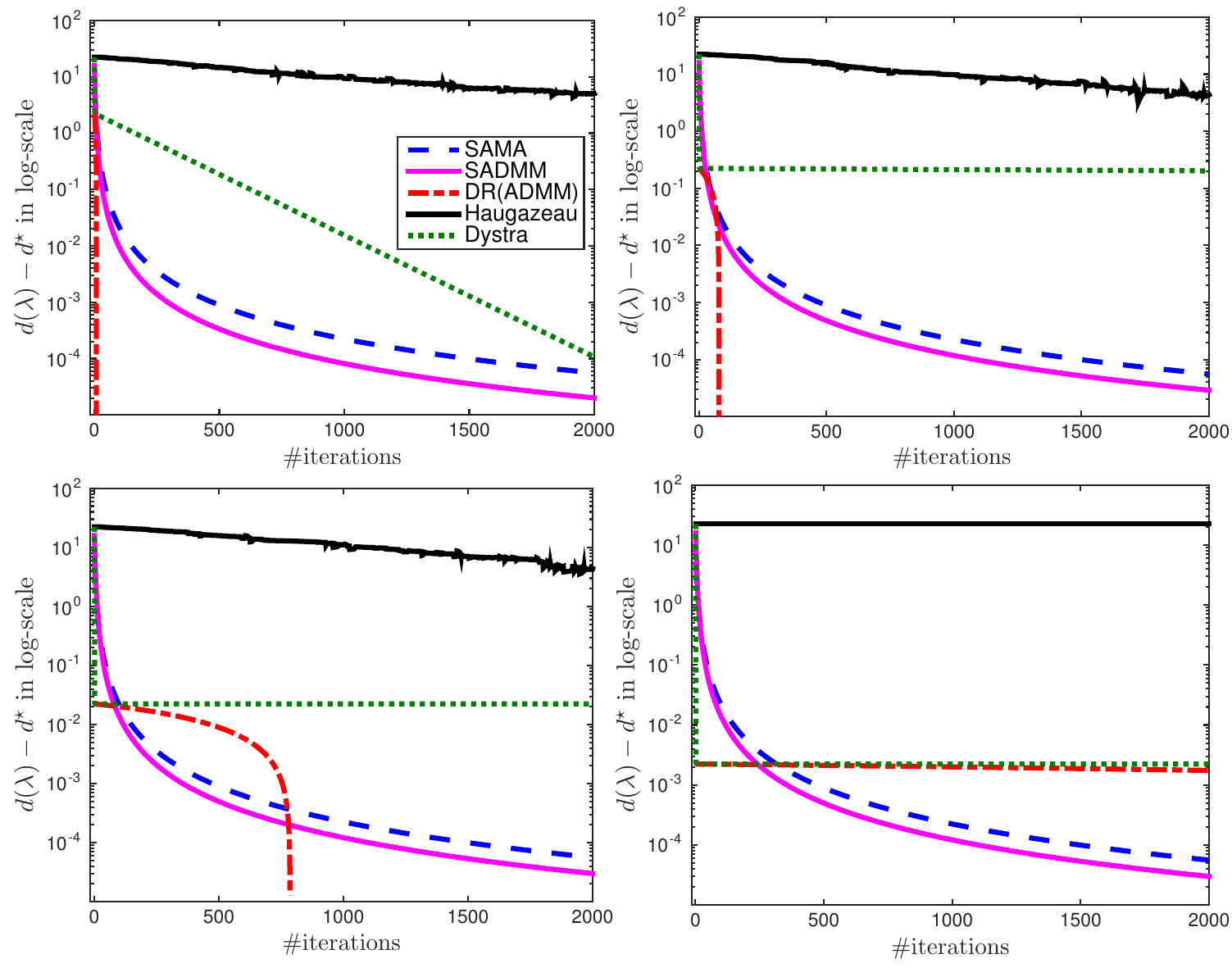} 
\vspace{-1ex}
\caption{The convergence behavior of five algorithms with different values of $\epsilon$.
These plots correspond to $\epsilon = 10^{-1}, 10^{-2}, 10^{-3}$ or $10^{-4}$ }\label{fig:feasibility_prob}
\vspace{-4ex}
\end{center}
\end{figure}

We observe Hauzageau's and Dykstra's methods are slow, but Hauzageau's method is extremely slow. 
The speed of ADMM (or DR splitting) strongly depends on the geometry of the sets, in particular, the tangent angle between two sets.
For large values of $\epsilon$, these methods work well, but they become arbitrarily slow when $\epsilon$ is decreasing.
The objective value of this method drops quickly to a certain level, then is saturated, and makes a very slow progress toward to the optimal value as seen in Figure~\ref{fig:feasibility_prob}.
Since the ADMM scheme  \eqref{eq:admm_fb}  is independent of its penalty parameter, this is the best performance we can achieve for solving \eqref{eq:distance problem}.
Both SAMA and SADMM have almost identical convergence rate for different values of $\epsilon$. 
These convergence rate reflects the theoretical guarantee, which is $\mathcal{O}(1/k)$ as predicted by our theoretical results.

\beforesec
\section{Discussion}\label{sec:conclusion}
\aftersec
We have developed a rigorous alternating direction optimization framework for solving constrained convex optimization problems. 
Our approach is built upon the model-based gap reduction (MGR) technique  in \cite{TranDinh2014b}, and unifies five main ideas: smoothing, gap reduction, alternating direction, acceleration/averaging, and homotopy.   
By splitting the gap, we have developed two new smooth alternating optimization algorithms: \ref{eq:new_ama_alg}  and \ref{eq:new_admm} with rigorous convergence guarantees. 
One  important feature of these methods is a heuristic-free parameter update, which has not been proved yet in the literature for AMA and ADMM as we discuss below: 

\vspace{1ex}
\noindent\textbf{(a)~Alternating direction method of multipliers $(\mathrm{ADMM})$:} 
The ADMM algorithm can be viewed as the Douglas-Rachford splitting applied to the optimality condition of the dual problem \eqref{eq:Fenchel_dual_prob}.  
As a result, the standard ADMM algorithm generates a primal sequence $\set{(u^k,v^k)}$ together with a multiplier sequence $\set{\lambda^k}$ as
\vspace{-0.75ex}
\begin{equation}\label{eq:admm_scheme}
\left\{\begin{array}{ll}
{\ub}^{k+1}  & := \displaystyle\mathrm{arg}\!\!\!\!\!\!\min_{\ub\in\dom{g}}\big\{g(\ub) - \iprods{{\lbd}^k, \Ab\ub}  + \frac{\eta_k}{2}\norm{\Ab\ub + \Bb{\vb}^k - \cb}^2_2\big\}\\
{\vb}^{k+1}  & := \displaystyle\mathrm{arg}\!\!\!\!\!\!\min_{\vb\in\dom{h}}\!\set{ h(\vb) - \iprods{{\lbd}^k, \Bb\vb} +  \frac{\eta_k}{2}\norm{\Ab{\ub}^{k + 1} + \Bb\vb - \cb}^2_2}\\
{\lbd}^{k+1} & := {\lbd}^k - \eta_k\big(\Ab\ub^{k+1} + \Bb\vb^{k+1} - \cb\big),
\end{array}\right.
\vspace{-0.75ex}
\end{equation}
where  $k$ denotes the iteration count and $\eta_k > 0$ is a penalty parameter. This basic method is closely related to or equivalent to many other algorithms, such as Spingarn's method of partial inverses, Dykstra's alternating projections, Bregman's iterative algorithms, and can also be motivated from the augmented Lagrangian perspective \cite{Boyd2011}. 

The ADMM algorithm serves as a good general-purpose tool for optimization problems arising in the analysis and processing of modern massive datasets. Indeed, its implementations have received a significant amount of engineering effort both in research and in industry. As a result, its global convergence rate characterizations for the template \eqref{eq:constr_cvx} is an active research topic, see, e.g., \cite{Chambolle2011,Davis2014,Davis2014b,Davis2015,Eckstein1992,Goldstein2012,He2012,Lin2013extragradient,Ouyang2014,Shefi2014,Wang2013a}, and the references quoted therein. 

In the constrained setting of \eqref{eq:constr_cvx}, a global convergence characterization specifically means the following: The algorithm provides us $\bar{\xb}^k = (\bar{\ub}^k,\bar{\vb}^k)$ and we determine the number of iterations $k$ necessary to obtain $f(\bar{\xb}^k) - \fopt \le \epsilon_f$ and $\norm{\Ab\bar{\ub}^k + \Bb\bar{\vb}^k - \cb} \le \epsilon_c$ for some fixed accuracy $\epsilon_f$ for the objective and for some--- possibly another---fixed accuracy $\epsilon_c$ for the linear constraint. Separating constraint feasibility is crucial so that the primal convergence has any significance otherwise we can trivially have $ \fopt-f(\bar{\xb}^k)  \le 0$ for some infeasible iterate  $\bar{\xb}^k$. 

A key theoretical strategy for obtaining global convergence rates for alternating direction methods is ergodic averaging 
\cite{Chambolle2011,Davis2014,Davis2014b,He2012,Lin2013extragradient,Nemirovskii2004,Ouyang2014,Shefi2014,Wang2013a}.
For instance, as opposed to working with the primal-sequence ${\xb}^k := ({\ub}^k,{\vb}^k)$ from \eqref{eq:admm_scheme} directly, we instead choose a sequence of weights $\set{\omega_k}\subset (0,+\infty)$ and then average as follows
\vspace{-1ex}
\begin{equation}\label{eq: weighting}
\bar{\xb}^k :=  \Big(\sum_{i=0}^k\omega_i\Big)^{-1}\sum_{i=0}^k\omega_i {\xb}^i.
\vspace{-1ex}
\end{equation}
The averaged sequence $\bar{\xb}^k$ then makes it theoretically elementary to obtain the desired type of convergence rate characterizations for \eqref{eq:constr_cvx}. 

Indeed, existing literature critically relies on such weighting strategies in order to obtain global convergence guarantees. For instance, He and Yuan in \cite{He2012} prove an $\mathcal{O}(1/k)$-convergence rate of their ADMM scheme \eqref{eq:admm_scheme}  by using the form \eqref{eq: weighting} with $\omega_i := 1$ but for both primal and dual variables $\xb$ as well as $\lambda$ simultaneously. They provided their guarantee in terms of a gap function for an associated variational inequality for \eqref{eq:constr_cvx} and assumed the boundedness on both primal and dual domains. 
This result is further extended by other authors to different variants of ADMM, including \cite{He2012a,Tao2012a,Wei20131}. The same rate is obtained in \cite{Davis2014b} for a relaxed ADMM variant with similar assumptions along with a weighting strategy that emphasizes the latter iterations by using $\omega_i := k+1$ in \eqref{eq: weighting}.  

We should note that there are also weighted global convergence characterizations for ADMM, such as $f(\bar{\xb}^k) - \fopt + \rho \norm{\Ab\bar{\ub}^k + \Bb\bar{\vb}^k - \cb}$ for some fixed $\rho>0$ by Shefi and Teboulle  \cite{Shefi2014}. 
The authors added proximal terms to the $\ub$- and $\vb$-subproblems and imposed conditions on three parameters to achieve the  $\mathcal{O}(1/k)$-convergence rate jointly between the objective residual and feasibility gap. Intriguingly, this type of convergence rate guarantee does not necessarily imply the $\mathcal{O}(1/k)$-convergence separately on the primal objective residual and feasibility gap as indicated in \cite[Theorem 5.2]{Shefi2014} without additional assumptions.

Interesingly, making additional assumptions on the template is quite common \cite{Davis2014b,Deng2012,Ghadimi2014,Goldstein2012}. For instance, the authors in \cite{Ouyang2014} studied a linearized ADMM variant  of \eqref{eq:admm_scheme} and proved the $\mathcal{O}(1/k)$-rate separately, but required the Lipschitz gradient assumption on either $g$ or $h$ in \eqref{eq:constr_cvx}. In addition, the authors in \cite{Goldstein2012} require strong convexity on both $g$ and $h$. In contrast, the authors \cite{Deng2012} require the strong convexity of either $g$ or $h$ but need $\Ab$ or $\Bb$ to be full rank as well. 
In \cite{Wei20131} the authors proposed an asynchronous ADMM and showed the $\mathcal{O}(1/k)$ rate on the averaging sequence for a special case of \eqref{eq:constr_cvx} where $h = 0$, which trivially has Lipschitz gradient.

Unsurprisingly, these assumptions again limit the applicability of the algorithmic guarantees when, for instance, $g$ and $h$ are non-Lipschitz loss functions or fully non-smooth regularizers, as in Poisson imaging, robust principal component analysis (RPCA), and graphical model learning \cite{Cevher2014}.
Several recent results rely on other type of assumptions such as error bounds, metric regularity, or the well-known Kurdyka-Lojasiewicz condition \cite{cai2016convergence,lin2015global,lin2015iteration}. 
Although these conditions cover a wide range of application models, it is unfortunately very hard to verify some quantities related to these assumptions in practice.
Other times, the additional assumptions obviate the ADMM choice as they can allow application of a simpler algorithm:

\vspace{1ex}
\noindent\textbf{(b)~Alternating minimization algorithm $(\mathrm{AMA})$:} 
The AMA algorithm, given below, is guaranteed to converge when $g$ is strongly convex or $g^{*}$ has Lipschitz gradient  \cite{Goldstein2012}: 
\vspace{-0.75ex}
\begin{equation}\label{eq:ama_scheme}
\left\{\begin{array}{ll}
\hat{\ub}^{k+1} &:= \displaystyle\mathrm{arg}\!\!\!\!\!\!\min_{\ub\in\dom{g}}\big\{ g(\ub) -  \iprods{\hat{\lbd}^k, \Ab\ub} \big\},\\
\hat{\vb}^{k + 1} & := \displaystyle\mathrm{arg}\!\!\!\!\!\!\min_{\vb\in\dom{h}}\big\{ h(\vb) - \iprods{\hat{\lbd}^k, \Bb\vb} + \frac{\eta_k}{2}\norm{\Ab\hat{\ub}^{k+1} + \Bb\vb - \cb}^2_2 \big\},\\
\hat{\lbd}^{k + 1} &  := \hat{\lbd}^k - \eta_k\big(\Ab\hat{\ub}^{k+1} + \Bb\hat{\vb}^{k+1} - \cb\big),
\end{array}\right.
\vspace{-0.75ex}
\end{equation}
where $\eta_k > 0$ is the penalty parameter. 
 
One can view AMA as the forward-backward splitting algorithm applied to the optimality condition of the dual problem \eqref{eq:Fenchel_dual_prob} (\textit{cf.}, \cite{Goldstein2012,Tseng1991}). Alternatively, we can motivate the algorithm by using one Lagrange dual step and one augmented Lagrangian dual step between two blocks of variables $\ub$ and $\vb$ \cite{bolte2014proximal,shefi2016rate,Tseng1991}. 
Computationally, \eqref{eq:ama_scheme} is arguably easier than \eqref{eq:admm_scheme}. 
However, it often requires stronger assumptions than ADMM to guarantee convergence \cite{Goldstein2012,Tseng1991}. The most obvious assumption is the strong convexity of $g$.

\beforesec
{\footnotesize
\section*{Acknowledgments} 
\aftersec
This work was supported in part by the NSF-grant No. DMS-1619884, USA; and 
MIRG-268398, ERC Future Proof and SNF 200021-132548, SNF 200021-146750 and SNF CRSII2-147633. 
The authors  would like to acknowledge Dr.\ C. B., Vu, and Dr.\ V. Q., Nguyen with their help on verifying the technical proofs and the numerical experiment. 
The authors also thank Mr. Ahmet Alacaoglu, Mr. Nhan Pham, and Ms. Yuzixuan Zhu for their careful proofreading. 
}

%

\beforesec
\section{The proof of technical results}
\aftersec
This appendix provides the full proof of technical results presented in the main text.

\beforesubsec
\subsection{The proof of Lemma \ref{le:optimal_bounds}: The primal-dual bounds}\label{apdx:le:optimal_bounds}
\aftersubsec
First, using the fact that $-d(\lbd) \leq -d^{\star} = \fopt  \leq \Lc(\xb, \lbd^{\star}) = f(\xb) + \iprods{\lbd^{\star}, \Ab\ub + \Bb\vb - \cb} \leq  f(\xb) + \norm{\lbd^{\star}}\norm{\Ab\ub + \Bb\vb - \cb}$, we get
\begin{equation}\label{eq:lm31_proof2}
-\norm{\lbd^{\star}}\norm{\Ab\ub + \Bb\vb - \cb} \leq f(\xb) - \fopt \leq f(\xb) + d(\lbd),
\end{equation}
which is exactly the lower bound \eqref{eq:lower_bound}.

Next, since $\Ab^{\top}\lbd^{\star}\in\partial{g}(\ub^{\star})$ due to \eqref{eq:opt_cond}, by Fenchel-Young's inequality, we have $g(\ub^{\star}) + g^{\ast}(\Ab^{\top}\lbd^{\star}) = \iprods{\Ab^{\top}\lbd^{\star}, \ub^{\star}}$, which implies $g^{\ast}(\Ab^{\top}\lbd^{\star}) = \iprods{\Ab^{\top}\lbd^{\star},\ub^{\star}} - g(\ub^{\star})$.
Using this relation and the definition of $\varphi_{\gamma}$, we have
\begin{align*}
\varphi_{\gamma}(\lbd) &:= \max\set{\iprods{\Ab^{\top}\lbd,\ub} - g(\ub) - \gamma b_{\Uc}(\ub,\bar{\ub}^c)} \geq \iprods{\Ab^{\top}\lbd,\ub^{\star}} - g(\ub^{\star}) - \gamma b_{\Uc}(\ub^{\star},\bar{\ub}^c)\notag\\
&=\iprods{\Ab^{\top}\lbd^{\star},\ub^{\star} } - g(\ub^{\star}) + \iprods{\Ab^{\top}(\lbd - \lbd^{\star}),\ub^{\star}} - \gamma b_{\Uc}(\ub^{\star},\bar{\ub}^c)\notag\\
&= g^{\ast}(\Ab^{\top}\lbd^{\star}) + \iprods{\Ab^{\top}(\lbd - \lbd^{\star}), \ub^{\star}} - \gamma b_{\Uc}(\ub^{\star},\bar{\ub}^c)\notag\\
&= \varphi(\lbd^{\star}) + \iprods{\lbd - \lbd^{\star},\Ab\ub^{\star}} - \gamma b_{\Uc}(\ub^{\star},\bar{\ub}^c).
\end{align*}
Alternatively, we have $\psi(\lbd) \geq \psi(\lbd^{\star}) + \iprods{\nabla{\psi}(\lbd^{\star}), \lbd - \lbd^{\star}}$, where $\nabla{\psi}(\lbd^{\star}) = \Bb\nabla{h^{\ast}}(\Bb^{\top}\lbd^{\star}) - \cb =  \Bb\vb^{\star}  - \cb$ due to the last relation in \eqref{eq:opt_cond}, where $\nabla{h^{\ast}}(\Bb^{\top}\lbd^{\star})\in\partial{h^{\ast}}(\Bb^{\top}\lbd^{\star})$ is one subgradient of $\partial{h^{\ast}}$.
Hence, $\psi(\lbd) \geq \psi(\lbd^{\star}) + \iprods{\lbd - \lbd^{\star}, \Bb\vb^{\star} - \cb}$.
Adding this inequality to the last estimation with the fact that $d_{\gamma} = \varphi_{\gamma} + \psi$ and $d = \varphi + \psi$, we obtain
\begin{equation}\label{eq:lm31_proof2b}
d_{\gamma}(\lbd) \geq d(\lbd^{\star}) + \iprods{\lbd - \lbd^{\star}, \Ab\ub^{\star} + \Bb\vb^{\star} - \cb} - \gamma b_{\Uc}(\ub^{\star},\bar{\ub}^c) \overset{\tiny\eqref{eq:opt_cond}}{=} d^{\star} - \gamma b_{\Uc}(\ub^{\star},\bar{\ub}^c)
\end{equation}
Using this inequality  with $d^{\star} = -f^{\star}$ and the definition \eqref{eq:G_gammabeta} of $f_{\beta}$ we have
\vspace{-0.75ex}
\begin{align}\label{eq:lm31_proof1}
f(\xb) - f^{\star} &\overset{\tiny\eqref{eq:G_gammabeta}+\eqref{eq:lm31_proof2b}}{\leq} f_{\beta}(\xb) + d_{\gamma}(\lbd) + \gamma b_{\Uc}(\ub^{\star},\bar{\ub}^c)  - \frac{1}{2\beta}\norm{\Ab\ub + \Bb\vb - \cb}^2\\
&= G_{\gamma\beta}(\wb) + \gamma b_{\Uc}(\ub^{\star},\bar{\ub}^c) - \frac{1}{2\beta}\norm{\Ab\ub +  \Bb\vb - \cb}^2.\nonumber
\vspace{-0.75ex}
\end{align}
Let $S := G_{\gamma\beta}(\wb) + \gamma b_{\Uc}(\ub^{\star},\bar{\ub}^c)$. 
Then, by dropping the last term $- \frac{1}{2\beta}\norm{\Ab\ub +  \Bb\vb - \cb}^2$ in \eqref{eq:lm31_proof1}, we obtain the first inequality of \eqref{eq:optimal_bounds}.

Let $t := \norm{\Ab\ub \!+\! \Bb\vb \!-\! \cb}$. 
Using again \eqref{eq:lm31_proof2} and \eqref{eq:lm31_proof1}, we can see that $\frac{1}{2\beta}t^2 - \norm{\lbd^{\star}}t - S \leq 0$. 
Solving this quadratic inequation w.r.t. $t$ and noting that $t \geq 0$, we obtain the second bound of  \eqref{eq:optimal_bounds}.
The last estimate of \eqref{eq:optimal_bounds} is a direct consequence of \eqref{eq:lm31_proof1}, the first one of \eqref{eq:optimal_bounds}.
Finally, from \eqref{eq:lm31_proof2}, we have $f(x) \geq f^{\star}  -\norm{\lbd^{\star}}\norm{\Ab\ub + \Bb\vb - \cb}$. Substituting this into \eqref{eq:lm31_proof1} we get $d(\lambda) - d^{\star} -  \norm{\lbd^{\star}}\norm{\Ab\ub + \Bb\vb - \cb} \leq S - \frac{1}{2\beta}\norm{\Ab\ub + \Bb\vb - \cb}^2$, which implies
\begin{equation*}
d(\lambda) - d^{\star} \leq S - (1/(2\beta))\norm{\Ab\ub + \Bb\vb - \cb}^2  + \norm{\lbd^{\star}}\norm{\Ab\ub + \Bb\vb - \cb}.
\end{equation*}
By discarding $- (1/(2\beta))\norm{\Ab\ub + \Bb\vb - \cb}^2$ and using the second estimate of \eqref{eq:optimal_bounds} into the last estimate, we obtain the last inequality of \eqref{eq:optimal_bounds}.
\Eproof

\beforesubsec
\subsection{Convergence analysis of Algorithm \ref{alg:SAMA}}
\aftersubsec
We provide a full proof of Lemmas and Theorems related to the convergence of Algorithm \ref{alg:SAMA}.
First, we prove the following key lemma, which will be used to prove Lemma \ref{le:ama_key_estimate2}.

\begin{lemma}\label{le:ama_key_estimate1b}
Let $\bar{\lbd}^{k\!+\!1}$ be  generated by  \eqref{eq:new_ama_alg}. Then
\begin{align}\label{eq:key_estimate1b}
d_{\gamma_{k\!+\!1}}(\bar{\lbd}^{k\!+\!1}) &{\!\!}\leq (1-\tau_k)d_{\gamma_{k\!+\!1}}(\bar{\lbd}^k) \!+\! \tau_k\hat{\ell}_{\gamma_{k\!+\!1}}(\lbd) \!+\! \tfrac{1}{\eta_k}\iprods{\bar{\lbd}^{k\!+\!1} \!\!-\! \hat{\lbd}^k, (1\!-\!\tau_k)\bar{\lbd}^k \!+\! \tau_k\lbd \!-\! \hat{\lbd}^k} \nonumber\\
&{\!\!\!} - \left(\tfrac{1}{\eta_k} \!-\! \tfrac{\norm{\Ab}^2}{2\gamma_{k\!+\!1}}\right)\Vert\bar{\lbd}^{k\!+\!1} \!-\! \hat{\lbd}^k\Vert^2 - \tfrac{(1-\tau_k)\gamma_{k\!+\!1}}{2}\Vert\ub^{\ast}_{\gamma_{k\!+\!1}}(\Ab^{\top}\bar{\lbd}^k) - \hat{\ub}^{k\!+\!1} \Vert^2,
\end{align}
where 
\begin{equation}\label{eq:hat_l_gamma}
\begin{array}{ll}
\hat{\ell}_{\gamma_{k\!+\!1}}(\lbd) &:= \varphi_{\gamma_{k\!+\!1}}(\hat{\lbd}^k) + \iprods{\nabla{\varphi_{\gamma_{k\!+\!1}}}(\hat{\lbd}^k), \lbd - \hat{\lbd}^k} + \psi(\lbd) \vspace{1ex}\\
& \leq  d_{\gamma_{k\!+\!1}}(\lbd) \!-\! \frac{\gamma_{k\!+\!1}}{2}\Vert\ub^{\ast}_{\gamma_{k\!+\!1}}(\Ab^{\top}\lbd) - \hat{\ub}^{k\!+\!1} \Vert^2.
\end{array}
\end{equation}
In addition, for any $\zb$, $\gamma_k, \gamma_{k\!+\!1} > 0$, the function $g_{\gamma}^{\ast}$ defined by  \eqref{eq:d1_gamma} satisfies 
\begin{equation}\label{eq:key_estimate1c}
g^{\ast}_{\gamma_{k\!+\!1}}(\zb) \leq g^{\ast}_{\gamma_k}(\zb) + (\gamma_k - \gamma_{k\!+\!1})b_{\Uc}(\ub^{\ast}_{\gamma_{k\!+\!1}}(\zb), \bar{\ub}^c).
\end{equation}
\end{lemma}

\begin{proof}
First, it is well-known that \ref{eq:new_ama_alg} is equivalent to the proximal-gradient step applying to the smoothed dual problem 
\begin{equation*}
\min_{\lbd}\set{ \varphi_{\gamma_{k+1}}(\lbd) + \psi(\lbd) : \lbd\in\R^n}.
\end{equation*} 
This proximal-gradient step can  be presented as
\begin{equation*}
\bar{\lbd}^{k\!+\!1} := \prox_{\eta_k\psi}\left(\hat{\lbd}^k - \eta_k\nabla{\varphi_{\gamma_{k\!+\!1}}}(\hat{\lbd}^k)\right).
\end{equation*}
We write down the optimality condition of this corresponding minimization problem of this step as 
\begin{equation*}
0 \in \partial{\psi}(\bar{\lbd}^{k\!+\!1}) + \nabla{\varphi_{\gamma_{k\!+\!1}}}(\hat{\lbd}^k) + \eta_k^{-1}(\bar{\lbd}^{k\!+\!1} - \hat{\lbd}^k).
\end{equation*}
Using this condition and the convexity of $\psi$, for any $\nabla{\psi}(\bar{\lbd}^{k\!+\!1})\in  \partial{\psi}(\bar{\lbd}^{k\!+\!1})$, we have 
\begin{align}\label{eq:lm41_est2}
\psi(\bar{\lbd}^{k\!+\!1})  &\leq  \psi(\lbd) + \iprods{\nabla{\psi}(\bar{\lbd}^{k\!+\!1}),\bar{\lbd}^{k\!+\!1} - \lbd} \notag\\
&= \psi(\lbd) +  \iprods{\nabla{\varphi_{\gamma_{k\!+\!1}}}(\hat{\lbd}^k),\lbd - \bar{\lbd}^{k\!+\!1}} + \eta_k^{-1}\iprods{\bar{\lbd}^{k\!+\!1} - \hat{\lbd}^k, \lbd - \bar{\lbd}^{k\!+\!1}}.
\end{align}
Next, by the definition $\varphi_{\gamma}(\lbd) := g^{\ast}_{\gamma}(\Ab^{\top}\lbd)$, we can show from \eqref{eq:d1_gamma} that $\hat{\ub}^{k\!+\!1} = \ub^{\ast}_{\gamma_{k\!+\!1}}(\Ab^{\top}\hat{\lbd}^k)$.
Since $g^{\ast}_{\gamma}$ is $(1/\gamma)$-Lipschitz gradient  continuous, we have 
\begin{equation*}
\frac{\gamma}{2}\norm{\nabla{g}^{\ast}_{\gamma}(\zb) - \nabla{g}^{\ast}_{\gamma}(\hat{\zb})}^2 \leq g^{\ast}_{\gamma}(\zb) - g^{\ast}_{\gamma}(\hat{\zb}) - \iprods{\nabla{g}^{\ast}_{\gamma}(\hat{\zb}), \zb - \hat{\zb}} \leq \frac{1}{2\gamma}\norm{\zb - \hat{\zb}}^2.
\end{equation*}
Using this inequality with $\gamma := \gamma_{k+1}$, $\nabla{g^{\ast}_{\gamma_{k+1}}}(\Ab^{\top}\lbd) = \ub^{\ast}_{\gamma_{k+1}}(\Ab^{\top}\lbd)$,  $\nabla{g^{\ast}_{\gamma_{k+1}}}(\Ab^{\top}\hat{\lbd}^k) = \ub^{\ast}_{\gamma_{k+1}}(\Ab^{\top}\hat{\lbd}^k) = \hat{\ub}^{k+1}$, and $\nabla{\varphi_{\gamma_{k+1}}}(\lbd) = \Ab\nabla{g^{\ast}_{\gamma_{k+1}}}(\Ab^{\top}\lbd)$, we have
\begin{equation}\label{eq:lm41_est1}
\begin{array}{ll}
\frac{\gamma_{k\!+\!1}}{2}\norm{ \ub^{\ast}_{\gamma_{k\!+\!1}}(\Ab^{\top}\lbd) - \hat{\ub}^{k\!+\!1}}^2 &\leq \varphi_{\gamma_{k\!+\!1}}(\lbd) - \varphi_{\gamma_{k\!+\!1}}(\hat{\lbd}^k)- \iprods{\nabla{\varphi}_{\gamma_{k\!+\!1}}(\hat{\lbd}^k), \lbd - \hat{\lbd}^{k}} \vspace{1ex}\\
&  \leq \frac{1}{2\gamma_{k+1}}\norm{\Ab^{\top}(\lbd - \hat{\lbd}^k)}^2 \leq \frac{\norm{\Ab}^2}{2\gamma_{k\!+\!1}}\norm{\lbd - \hat{\lbd}^k}^2.
\end{array}
\end{equation}
Using \eqref{eq:lm41_est1} with $\lbd = \bar{\lbd}^{k\!+\!1}$, we have
\begin{align*}
\varphi_{\gamma_{k\!+\!1}}(\bar{\lbd}^{k\!+\!1}) \leq \varphi_{\gamma_{k\!+\!1}}(\hat{\lbd}^k) + \iprods{\nabla{\varphi}_{\gamma_{k\!+\!1}}(\hat{\lbd}^k),\bar{\lbd}^{k\!+\!1} - \hat{\lbd}^k} + \frac{\norm{\Ab}^2}{2\gamma_{k\!+\!1}}\norm{\bar{\lbd}^{k\!+\!1} - \hat{\lbd}^k}^2.
\end{align*}
Summing up this inequality and \eqref{eq:lm41_est2}, then using the definition of $\hat{\ell}_{\gamma_{k+1}}(\lbd)$  in \eqref{eq:hat_l_gamma}, we obtain
\begin{equation}\label{eq:key_estimate1}
d_{\gamma_{k\!+\!1}}(\bar{\lbd}^{k\!+\!1}) \leq  \hat{\ell}_{\gamma_{k+1}}(\lbd) + \tfrac{1}{\eta_k}\iprods{\bar{\lbd}^{k\!+\!1} - \hat{\lbd}^k, \lbd - \hat{\lbd}^k} - \left(\tfrac{1}{\eta_k} - \tfrac{\norm{\Ab}^2}{2\gamma_{k\!+\!1}}\right)\norm{\bar{\lbd}^{k\!+\!1} - \hat{\lbd}^k}^2.
\end{equation}
Here, the second inequality in \eqref{eq:hat_l_gamma} follows from the right-hand side of \eqref{eq:lm41_est1}.

Now, using \eqref{eq:key_estimate1} with $\lbd := \bar{\lbd}^k$, then combining with  \eqref{eq:hat_l_gamma},  we get
\begin{align*}
\begin{array}{ll}
d_{\gamma_{k\!+\!1}}(\bar{\lbd}^{k\!+\!1}) &\leq d_{\gamma_{k\!+\!1}}(\bar{\lbd}^k) + \frac{1}{\eta_k}\iprods{\bar{\lbd}^{k\!+\!1} - \hat{\lbd}^k, \bar{\lbd}^k - \hat{\lbd}^k} -  \left(\frac{1}{\eta_k} - \frac{\norm{\Ab}^2}{2\gamma_{k\!+\!1}}\right)\norm{\bar{\lbd}^{k\!+\!1} - \hat{\lbd}^k}^2 \vspace{1ex}\\
&- \frac{\gamma_{k\!+\!1}}{2}\Vert\ub^{\ast}_{\gamma_{k\!+\!1}}(\Ab^{\top}\bar{\lbd}^k) - \hat{\ub}^{k\!+\!1} \Vert^2.
\end{array}
\end{align*}
Multiplying the last inequality by $1-\tau_k\in [0,1]$ and \eqref{eq:key_estimate1} by $\tau_k\in [0, 1]$, then summing up the results, we obtain \eqref{eq:key_estimate1b}.

Finally, from \eqref{eq:d1_gamma}, since $g^{\ast}_{\gamma}(\zb) := \max_{\ub}\set{P(\ub, \gamma; \zb) := \iprods{\zb, \ub} - g(\ub) - \gamma b_{\Uc}(\ub;\bar{\ub}^c)}$, is the maximization of $P$ over $\ub$ indexing in $\gamma$ and $\zb$, which is concave in $\ub$ and linear in $\gamma$, we have $g^{\ast}_{\gamma}(\zb)$ is convex w.r.t. $\gamma > 0$.
Moreover, $\frac{d g^{\ast}_{\gamma}(\zb)}{d\gamma} = -b_{\Uc}(\ub^{\ast}_{\gamma}(\zb), \bar{\ub}^c)$.
Hence, using the convexity of $g^{\ast}_{\gamma}$ w.r.t. $\gamma > 0$, we have $g^{\ast}_{\gamma_k}(\zb) \geq g^{\ast}_{\gamma_{k\!+\!1}}(\zb) - (\gamma_k - \gamma_{k\!+\!1})b_{\Uc}(\ub^{\ast}_{\gamma}(\zb), \bar{\ub}^c)$, which is indeed \eqref{eq:key_estimate1c}.
\end{proof}

\vspace{-2ex}
\beforesubsubsec
\subsubsection{The proof of Lemma \ref{le:init_point}: Bound on $G_{\gamma\beta}$ for the first iteration}\label{apdx:le:init_point}
\aftersubsubsec
Since $\bar{\wb}^1 := (\bar{\ub}^1, \bar{\vb}^1, \bar{\lbd}^1)$ is updated by \eqref{eq:init_point}, similar to \eqref{eq:new_ama_alg}, we can use \eqref{eq:key_estimate1} with $k=0$, $\lbd := \hat{\lbd}^0$ and $\hat{\ell}_{\gamma_1}(\hat{\lbd}^0) \leq d_{\gamma_1}(\hat{\lbd}^0)$ to obtain
\begin{align}\label{eq:lm34_proof1}
d_{\gamma_1}(\bar{\lbd}^1) \leq  d_{\gamma_1}(\hat{\lbd}^0) - \left(\frac{1}{\eta_0} - \frac{\norm{\Ab}^2}{2\gamma_1}\right)\norm{\bar{\lbd}^1 - \hat{\lbd}^0}^2.
\end{align}
Since $\bar{\vb}^1$ solves the second problem in  \eqref{eq:init_point} and $\vb^{\ast}(\hat{\lbd}^0) \in \dom{h}$, we have 
\begin{align*} 
\begin{array}{ll}
h(\vb^{\ast}(\hat{\lbd}^0)) &- \iprods{\hat{\lbd}^0,\Bb\vb^{\ast}(\hat{\lbd}^0)} + \frac{\eta_0}{2}\norm{\Ab\bar{\ub}^1 + \Bb\vb^{\ast}(\hat{\lbd}^0) - \cb}^2 \geq h(\bar{\vb}^1) \vspace{1ex}\\
& - \iprods{\hat{\lbd}^0,\Bb\bar{\vb}^1} + \frac{\eta_0}{2}\norm{\Ab\bar{\ub}^1 + \Bb\bar{\vb}^1 - \cb}^2 + \frac{\eta_0}{2}\norm{\Bb(\vb^{\ast}(\hat{\lbd}^0) - \bar{\vb}^1)}^2. 
\end{array}
\end{align*}
Using $D_f$ in \eqref{eq:Df_diameter}, this inequality implies
\begin{equation}\label{eq:lm34_proof2}
h^{\ast}(\Bb^{\top}\hat{\lbd}^0) \leq \iprods{\hat{\lbd}^0,\Bb\bar{\vb}^1}  - h(\bar{\vb}^1) -   \frac{\eta_0}{2}\norm{\Ab\bar{\ub}^1 \!\!+\! \Bb\bar{\vb}^1 \!\!-\! \cb}^2 + \frac{\eta_0}{2}\norm{\Ab\bar{\ub}^1 \!\!+\! \Bb\bar{\vb}^1 \!\!-\! \cb}D_f.
\end{equation}
Using the definition of $d_{\gamma}$, we further estimate \eqref{eq:lm34_proof1} using \eqref{eq:lm34_proof2} as follows:
\begin{equation*} 
{}\begin{array}{ll}
d_{\gamma_1}(\bar{\lbd}^1) & \overset{\tiny\eqref{eq:lm34_proof1}}{\leq}{\!\!} \varphi_{\gamma_1}(\hat{\lbd}^0) + \psi(\hat{\lbd}^0) - \left(\frac{1}{\eta_0} - \frac{\norm{\Ab}^2}{2\gamma_1}\right)\norm{\bar{\lbd}^1 - \hat{\lbd}^0}^2 \vspace{0.75ex}\\
&  \overset{\tiny\eqref{eq:d1_gamma}}{=} \iprods{\Ab\bar{\ub}^1, \hat{\lbd}^0} \!-\! g(\bar{\ub}^1) \!-\! \gamma_1b_{\Uc}(\bar{\ub}^1,\bar{\ub}^c) \!+\! \psi(\hat{\lbd}^0) \!-\! \left(\frac{1}{\eta_0} \!-\! \frac{\norm{\Ab}^2}{2\gamma_1}\right)\norm{\bar{\lbd}^1 \!-\! \hat{\lbd}^0}^2\vspace{1ex}\\
& \overset{\tiny\eqref{eq:lm34_proof2}}{{\!\!}\leq{\!\!}}\iprods{\hat{\lbd}^0, \Ab\bar{\ub}^1 \!+\! \Bb\bar{\vb}^1 \!-\! \cb} -  g(\bar{\ub}^1) \!-\! h(\bar{\vb}^1) -  \gamma_1b_{\Uc}(\bar{\ub}^1,\bar{\ub}^c)\vspace{1ex}\\
& -   \frac{\eta_0}{2}\norm{\Ab\bar{\ub}^1 \!+\! \Bb\bar{\vb}^1 \!-\! \cb}^2 \!-\! \left(\frac{1}{\eta_0} \!-\! \frac{\norm{\Ab}^2}{2\gamma_1}\right)\norm{\bar{\lbd}^1 \!-\! \hat{\lbd}^0}^2  \!+\! \frac{\eta_0}{2}\norm{\Ab\bar{\ub}^1 \!+\! \Bb\bar{\vb}^1 \!\!-\! \cb}D_f \vspace{1ex}\\
& \leq -f_{\beta_1}(\bar{\xb}^1) \!+\! \frac{1}{2\eta_0^2}\Big[\frac{1}{\beta_1} \!-\! \frac{5\eta_0}{2} \!+\! \frac{\norm{\Ab}^2\eta_0^2}{\gamma_1}\Big]\norm{\bar{\lbd}^1 \!-\! \hat{\lbd}^0}^2  + \frac{1}{\eta_0}\iprods{\hat{\lbd}^0, \bar{\lbd}^1 - \hat{\lbd}^0} + \frac{\eta_0}{4}D_f^2.
\end{array}
\end{equation*}
Since $G_{\gamma_1\beta_1}(\bar{\wb}^1) = f_{\beta_1}(\bar{\xb}^1) + d_{\gamma_1}(\bar{\lbd}^1)$, we obtain \eqref{eq:init_point_est} from the last inequality.
If $\beta_1 \geq \frac{2\gamma_1}{\eta_0(5\gamma_1 - 2\norm{\Ab}^2\eta_0)}$, then \eqref{eq:init_point_est} leads to $G_{\gamma_1\beta_1}(\bar{\wb}^1) \leq \frac{\eta_0}{4}D_f^2 + \frac{1}{\eta_0}\iprods{\hat{\lbd}^0, \bar{\lbd}^1 - \hat{\lbd}^0}$.
\Eproof

\beforesubsubsec
\subsubsection{The proof of Lemma \ref{le:ama_key_estimate2}: Gap reduction condition}\label{apdx:le:ama_key_estimate2}
\aftersubsubsec
For notational simplicity, we first define the following abbreviations
\begin{equation*}
\left\{\begin{array}{ll}
\bar{z}^k &:= \Ab\bar{\ub}^k + \Bb\bar{\vb}^k - \cb  \vspace{0.5ex}\\
\hat{\zb}^{k\!+\!1} &:= \Ab\hat{\ub}^{k\!+\!1} \!+\! \Bb\hat{\vb}^{k\!+\!1} \!-\! \cb \vspace{0.5ex}\\
\bar{\ub}_{k\!+\!1}^{*} &:= \ub^{*}_{\gamma_{k\!+\!1}}(\Ab^{\top}\bar{\lbd}^k)~~\text{the solution of \eqref{eq:d1_gamma} at $\bar{\lbd}^k$,} \vspace{0.5ex}\\
\hat{v}^{*}_k &:= v^{*}(\hat{\lbd}^k) \in\partial{h^{\ast}}(\Ab^{\top}\hat{\lbd}^k) ~~\text{a subgradient of $h^{\ast}$ defined by \eqref{eq:dual_func_i} at $\Ab^{\top}\hat{\lbd}^k$, and}\vspace{0.5ex}\\
D_k &:= \Vert\Ab\hat{\ub}^{k\!+\!1} + \Bb(2\hat{\vb}^{*}_k - \hat{\vb}^{k\!+\!1}) - \cb\Vert. 
\end{array}\right.
\end{equation*}
From \ref{eq:new_ama_alg}, we have $\bar{\lbd}^{k\!+\!1} - \hat{\lbd}^k = \eta_k(\cb - \Ab\hat{\ub}^{k\!+\!1} - \Bb\hat{\vb}^{k\!+\!1}) = -\eta_k\hat{\zb}^{k\!+\!1}$. 
In addition, by \eqref{eq:add_ama_steps}, we have $\hat{\lbd}^k = (1-\tau_k)\bar{\lbd}^k + \tau_k\lbd_k^{*}$, which leads to 
$(1-\tau_k)\bar{\lbd}^k + \tau_k\hat{\lbd}^k - \hat{\lbd}^k = \tau_k(\hat{\lbd}^k - \lbd_k^{*})$. 
Using these expressions into \eqref{eq:key_estimate1b} with $\lbd := \hat{\lbd}^k$, and then using \eqref{eq:hat_l_gamma} with $\hat{\ell}_{\gamma_{k+1}}(\hat{\lbd}^k) \leq d_{\gamma_{k+1}}(\hat{\lbd}^k)$, we obtain
\vspace{-0.75ex}
\begin{equation}\label{eq:ama_lemma41_est1}
\begin{array}{ll}
d_{\gamma_{k\!+\!1}}(\bar{\lbd}^{k\!+\!1}) &\leq (1-\tau_k)d_{\gamma_{k\!+\!1}}(\bar{\lbd}^k) + \tau_kd_{\gamma_{k\!+\!1}}(\hat{\lbd}^k) + \tau_k\iprods{\hat{\zb}^{k\!+\!1}, \lbd_k^{*} - \hat{\lbd}^k} \vspace{1ex}\\
& - \eta_k\left(1 - \frac{\eta_k\norm{\Ab}^2}{2\gamma_{k\!+\!1}}\right)\Vert\hat{\zb}^{k\!+\!1}\Vert^2 - (1-\tau_k)\frac{\gamma_{k\!+\!1}}{2}\norm{\bar{\ub}^{\ast}_{k\!+\!1} - \hat{\ub}^{k\!+\!1}}^2.
\end{array}
\vspace{-0.75ex}
\end{equation}
By \eqref{eq:key_estimate1c} with the fact that $\varphi_{\gamma}(\lbd) := g^{\ast}_{\gamma}(\Ab^{\top}\lbd)$, for any $\gamma_{k\!+\!1} > 0$ and $\gamma_k > 0$, we have
\vspace{-0.75ex}
\begin{equation*} 
\varphi_{\gamma_{k\!+\!1}}(\bar{\lbd}^k) \leq \varphi_{\gamma_k}(\bar{\lbd}^k) + (\gamma_k - \gamma_{k\!+\!1})b_{\Uc}(\bar{\ub}_{k\!+\!1}^{\ast}, \bar{\ub}_c).
\vspace{-0.75ex}
\end{equation*}
Using this inequality and the fact that $d_{\gamma} := \varphi_{\gamma} + \psi$, we have
\begin{equation}\label{eq:ama_lemma41_est2}
d_{\gamma_{k\!+\!1}}(\bar{\lbd}^k) \leq d_{\gamma_k}(\bar{\lbd}^k)  +  (\gamma_k - \gamma_{k\!+\!1})b_{\Uc}(\bar{\ub}_{k\!+\!1}^{\ast}, \bar{\ub}_c).
\end{equation}
Next, using $\hat{\vb}^{k\!+\!1}$ from \ref{eq:new_ama_alg} and its optimality condition, we can show that
\begin{equation*}
{\!}\begin{array}{ll}
h^{\ast}(\Bb^{\top}\hat{\lbd}^k) & - \! \frac{\eta_k}{2}\Vert\Ab\hat{\ub}^{k\!+\!1} \!+\! \Bb\hat{\vb}^{*}_k \!-\! \cb\Vert^2 =  \iprods{\Bb^{\top}\hat{\lbd}^k, \hat{\vb}^{*}_k} - h(\hat{\vb}^{*}_k) - \frac{\eta_k}{2}\Vert\Ab\hat{\ub}^{k\!+\!1} + \Bb\hat{\vb}^{*}_k - \cb\Vert^2 \vspace{1ex}\\
&\leq \iprods{\Bb^{\top}\hat{\lbd}^k, \hat{\vb}^{k\!+\!1}} - h(\hat{\vb}^{k\!+\!1}) - \frac{\eta_k}{2}\Vert\Ab\hat{\ub}^{k\!+\!1} \!+\! \Bb\hat{\vb}^{k\!+\!1} \!-\! \cb\Vert^2 - \frac{\eta_k}{2}\Vert\Bb(\hat{\vb}^{*}_k - \hat{\vb}^{k\!+\!1})\Vert^2.
\end{array}
\end{equation*}
Since $\psi(\lbd) := h^{\ast}(\Bb^{\top}\lbd) - \cb^{\top}\lbd$, this inequality leads to
\begin{equation*}
\begin{array}{ll}
\psi(\hat{\lbd}^k) &{\!\!}\leq \iprods{\Bb^{\top}\hat{\lbd}^k, \hat{\vb}^{k\!+\!1}} \!-\! \iprods{\cb,\hat{\lbd}^k} \!-\! h(\hat{\vb}^{k\!+\!1}) \!-\! \frac{\eta_k}{2}\Vert\hat{\zb}^{k\!+\!1}\Vert^2 \!-\! \frac{\eta_k}{2}\iprods{\hat{\zb}^{k\!+\!1}, \Ab\hat{\ub}^{k\!+\!1} \!+\! \Bb(2\hat{\vb}^{*}_k \!-\! \hat{\vb}^{k\!+\!1}) \!-\! c}\vspace{1ex}\\
&{\!\!} \leq  \iprods{\hat{\lbd}^k, \Bb\hat{\vb}^{k\!+\!1}-\cb} - h(\hat{\vb}^{k\!+\!1}) -  \frac{\eta_k}{2}\Vert\hat{\zb}^{k\!+\!1}\Vert^2 + \frac{\eta_k}{2}\norm{\hat{\zb}^{k\!+\!1}}D_k.
\end{array}
\end{equation*}
Now, by this estimate, $d_{\gamma_{k\!+\!1}} = \varphi_{\gamma_{k\!+\!1}}  + \psi$ and \ref{eq:new_ama_alg}, we can derive
\begin{equation*}
{\!\!}\begin{array}{ll}
d_{\gamma_{k\!+\!1}}(\hat{\lbd}^k) &{\!}\leq \varphi_{\gamma_{k\!+\!1}}(\hat{\lbd}^k) \!-\! h(\hat{\vb}^{k\!+\!1}) \!+\! \iprods{\hat{\lbd}^k, \Bb\hat{\vb}^{k\!+\!1} - \cb} \!-\! \frac{\eta_k}{2}\Vert\hat{\zb}^{k\!+\!1}\Vert^2 \!+\! \frac{\eta_k}{2}\norm{\hat{\zb}^{k\!+\!1}}D_k{\!}\vspace{1ex}\\
&{\!} = -f(\hat{\xb}^{k\!+\!1}) + \iprods{\hat{\lbd}^k, \hat{\zb}^{k\!+\!1}} - \frac{\eta_k}{2}\Vert\hat{\zb}^{k\!+\!1}\Vert^2 + \frac{\eta_k}{2}\norm{\hat{\zb}^{k\!+\!1}}D_k - \gamma_{k\!+\!1}b_{\Uc}(\hat{\ub}^{k\!+\!1}, \bar{\ub}^c).
\end{array}{}
\end{equation*}
Combining this inequality, \eqref{eq:ama_lemma41_est1} and \eqref{eq:ama_lemma41_est2}, we obtain
\begin{equation}\label{eq:ama_lemma41_est4}
{\!\!\!}\begin{array}{ll}
d_{\gamma_{k\!+\!1}}\!(\bar{\lbd}^{k\!+\!1}) &{\!\!}\leq (1\!-\!\tau_k)d_{\gamma_k}(\bar{\lbd}^k) \!-\!
 \tau_kf(\hat{\xb}^{k\!+\!1}) \!+\! \tau_k\iprods{\lbd_k^{*}, \hat{\zb}^{k\!+\!1}}   \vspace{1ex}\\
& {\!} - \eta_k{\!}\left(1 \!+\! \frac{\tau_k}{2} \!-\! \frac{\norm{\Ab}^2\eta_k}{2\gamma_{k\!+\!1}}\right){\!}\Vert\hat{\zb}^{k\!+\!1}\Vert^2\vspace{1ex}\\
&{\!}- \tau_k\gamma_{k\!+\!1}b_{\Uc}(\hat{\ub}^{k\!+\!1}, \bar{\ub}^c) \!+\! (1\!-\!\tau_k)(\gamma_k - \gamma_{k\!+\!1})b_{\Uc}(\bar{\ub}_{k\!+\!1}^{\ast}, \bar{\ub}_c) \vspace{1ex}\\
&{\!} - (1-\tau_k)\frac{\gamma_{k\!+\!1}}{2}\norm{\bar{\ub}^{\ast}_{k\!+\!1} - \hat{\ub}^{k\!+\!1}}^2 +  \frac{\tau_k\eta_k}{2}\norm{\hat{\zb}^{k\!+\!1}}D_k.
\end{array}{\!\!}
\end{equation}
Now, using the definition $G_k$, we have
\begin{equation*}
\begin{array}{ll}
G_k(\bar{\wb}^k) &:= f_{\beta_k}(\bar{\xb}^k) + d_{\gamma_k}(\bar{\lbd}^k) = f(\bar{\xb}^k) + d_{\gamma_k}(\bar{\lbd}^k) + \frac{1}{2\beta_k}\Vert\Ab\bar{\ub}^k + \Bb\bar{\vb}^k - \cb\Vert^2 \vspace{0.5ex}\\
& = f(\bar{\xb}^k) + d_{\gamma_k}(\bar{\lbd}^k) + \frac{1}{2\beta_k}\Vert\bar{\zb}^k\Vert^2.
\end{array}
\end{equation*}
Let us define $\Delta{G}_k := (1-\tau_k)G_k(\bar{\wb}^k) - G_{k\!+\!1}(\bar{\wb}^{k\!+\!1})$. 
Then, we can show that 
\begin{equation}\label{eq:ama_lemma41_est5}
\begin{array}{ll}
\Delta{G}_k &= (1-\tau_k)f(\bar{\xb}^k) + (1-\tau_k)d_{\gamma_k}(\bar{\lbd}^k) - f(\bar{\xb}^{k\!+\!1})  - d_{\gamma_{k\!+\!1}}(\bar{\lbd}^{k\!+\!1})  \vspace{1ex}\\
& +  \frac{(1-\tau_k)}{2\beta_k}\Vert \bar{\zb}^k\Vert^2 - \frac{1}{2\beta_{k\!+\!1}}\Vert\bar{\zb}^{k\!+\!1}\Vert^2.
\end{array}
\end{equation}
By \eqref{eq:add_ama_steps}, we have $\bar{\zb}^{k\!+\!1} = (1-\tau_k)\bar{\zb}^k + \tau_k\hat{\zb}^{k\!+\!1}$. Using this expression and  the condition $\beta_{k\!+\!1} \geq (1-\tau_k)\beta_k$ in \eqref{eq:ama_param_cond}, we can easily show that
\begin{equation*}
\frac{(1 \!-\! \tau_k)}{2\beta_k}\Vert\bar{\zb}^k\Vert^2 - \frac{1}{2\beta_{k\!+\!1}}\Vert\bar{\zb}^{k\!+\!1}\Vert^2 \geq 
- \frac{\tau_k}{\beta_{k}}\iprods{\hat{\zb}^{k\!+\!1}, \bar{\zb}^k}  - \frac{\tau_k^2}{2\beta_{k}(1-\tau_k)}\Vert\hat{\zb}^{k\!+\!1}\Vert^2.
\end{equation*}
Substituting this inequality into \eqref{eq:ama_lemma41_est5}, and using the convexity of $f$, we further get
\begin{equation}\label{eq:ama_lemma41_est5b}
\begin{array}{ll}
{\!\!}\Delta{G}_k &\geq (1 \!-\! \tau_k)d_{\gamma_k}(\bar{\lbd}^k) - d_{\gamma_{k\!+\!1}}(\bar{\lbd}^{k\!+\!1}) \!-\! \tau_kf(\hat{\xb}^{k\!+\!1}) \vspace{0.75ex}\\
& - \frac{\tau_k}{\beta_k}\iprods{\hat{\zb}^{k\!+\!1}, \bar{\zb}^k}  - \frac{\tau_k^2}{2(1 \!-\! \tau_k)\beta_k}\Vert\hat{\zb}^{k\!+\!1}\Vert^2.{\!\!}
\end{array}
\end{equation}
Substituting \eqref{eq:ama_lemma41_est4} into \eqref{eq:ama_lemma41_est5b} and using $\lbd^{*}_k := \frac{1}{\beta_k}(\cb - \Ab\bar{\ub}^k - \Bb\bar{\vb}^k) = -\frac{1}{\beta_k}\bar{\zb}^k$, we obtain
\begin{equation}\label{eq:ama_lemma41_est6}
\Delta{G}_k \geq \Big[ \eta_k\Big(1 + \frac{\tau_k}{2} - \frac{\norm{\Ab}^2\eta_k}{2\gamma_{k\!+\!1}}\Big) - \frac{\tau_k^2}{2(1-\tau_k)\beta_k}\Big]\Vert\hat{\zb}^{k\!+\!1}\Vert^2 + R_k - \frac{\tau_k\eta_k}{2}\norm{\hat{\zb}^{k\!+\!1}}D_k.
\end{equation}
where 
\begin{equation*}
R_k :=  \frac{(1-\tau_k)}{2}\gamma_{k\!+\!1}\norm{\bar{\ub}^{\ast}_{k\!+\!1}-\hat{\ub}^{k\!+\!1}}^2 + \tau_k\gamma_{k\!+\!1}b_{\Uc}(\hat{\ub}^{k\!+\!1}, \bar{\ub}^c)  - (1-\tau_k)(\gamma_k - \gamma_{k\!+\!1})b_{\Uc}(\bar{\ub}^{\ast}_{k\!+\!1}, \bar{\ub}^c).
\end{equation*}
Furthermore, we have
\begin{equation*}
\frac{\eta_k}{4}\norm{\hat{\zb}^{k\!+\!1}}^2 - \frac{\tau_k\eta_k}{2}\norm{\hat{\zb}^{k\!+\!1}}D_k 
= \frac{\eta_k}{4}\big[\norm{\zb^{k\!+\!1}} - \tau_kD_k\big]^2 -  \frac{\eta_k\tau_k^2D_k^2}{4}  \geq - \frac{\eta_k\tau_k^2D_k^2}{4}.
\end{equation*} 
Using this estimate into \eqref{eq:ama_lemma41_est6}, we finally get
\begin{equation}\label{eq:ama_lemma41_est6b}
\Delta{G}_k  \geq \Big[ \eta_k\Big(\frac{3}{4} + \frac{\tau_k}{2} - \frac{\norm{\Ab}^2\eta_k}{2\gamma_{k\!+\!1}}\Big) - \frac{\tau_k^2}{2(1-\tau_k)\beta_k}\Big]\Vert\hat{\zb}^{k\!+\!1}\Vert^2 + R_k -  \frac{\eta_k\tau_k^2D_k^2}{4}.
\end{equation}
Next step, we  estimate $R_k$.
Let $\bar{a}_k := \bar{\ub}^{*}_{k\!+\!1} - \bar{\ub}_c$, $\hat{a}_k := \hat{\ub}^{k\!+\!1} - \bar{\ub}_c$. 
Using the smoothness of $b_{\Uc}$, we can estimate $R_k$ explicitly as
\begin{equation}\label{eq:Rk_est}
\begin{array}{ll}
2\gamma_{k\!+\!1}^{-1}R_k &{\!} \geq  (1-\tau_k)\norm{\bar{a}_k - \hat{a}_k}^2 - (1-\tau_k)(\gamma_{k\!+\!1}^{-1}\gamma_{k} - 1)L_b\norm{\bar{a}_k}^2 + \tau_k\norm{\hat{a}_k}^2\vspace{1ex}\\
&{\!} = \Vert\hat{a}^k - (1-\tau_k)\bar{a}_k\Vert^2 + (1-\tau_k)\left(\tau_k - (\gamma_{k\!+\!1}^{-1}\gamma_{k} - 1)L_b\right)\Vert\bar{a}_k\Vert^2.
\end{array}
\end{equation}
By the condition $(1+L_b^{-1}\tau_k)\gamma_{k\!+\!1} \geq \gamma_k$ in \eqref{eq:ama_param_cond}, we have $\tau_k - (\gamma_{k\!+\!1}^{-1}\gamma_{k} - 1)L_b\geq 0$. 
Using this condition in \eqref{eq:Rk_est}, we obtain $R_k \geq 0$. 
Finally, by \eqref{eq:Df_diameter} we can show that $D_k \leq D_f$. 
Using this inequality, $R_k\geq 0$, and the second condition of \eqref{eq:ama_param_cond}, we can show from \eqref{eq:ama_lemma41_est6} that 
$\Delta{G}_k \geq -\frac{\eta_k\tau_k^2}{4}D_f^2$, which implies \eqref{eq:ama_key_estimate2}.
\Eproof

\beforesubsubsec
\subsubsection{The proof of Lemma~\ref{le:update_rules_ama}: Parameter updates}\label{apdx:le:update_rules_ama}
\aftersubsubsec
The tightest update for $\gamma_k$ and $\beta_k$ is $\gamma_{k\!+\!1} := \frac{\gamma_k}{\tau_k+1}$ and $\beta_{k\!+\!1} := (1-\tau_k)\beta_k$ due to \eqref{eq:ama_param_cond}.
Using these updates in the third condition in \eqref{eq:ama_param_cond} leads to $\frac{(1-\tau_{k\!+\!1})^2}{(1+\tau_{k\!+\!1})\tau_{k\!+\!1}^2} \geq \frac{1-\tau_k}{\tau_k^2}$.
By directly checking this condition, we can see that $\tau_k = \mathcal{O}(1/k)$ which is the optimal choice.

Clearly, if we choose $\tau_k := \frac{3}{k+4}$, then $0 < \tau_k < 1$ for $k\geq 0$ and $\tau_0 = 3/4$.
Next, we choose $\gamma_{k\!+\!1} := \frac{\gamma_k}{1+\tau_k/3} \geq \frac{\gamma_k}{1+\tau_k}$. 
Substituting $\tau_k = \frac{3}{k+4}$ into this formula we have $\gamma_{k\!+\!1} = \left(\frac{k+4}{k+5}\right)\gamma_k$.
By induction, we obtain $\gamma_{k\!+\!1} = \frac{5\gamma_1}{k+5}$. 
This implies $\eta_k = \frac{5\gamma_1}{2\norm{\Ab}^2(k+5)}$.
With $\tau_k = \frac{3}{k+4}$ and $\gamma_{k\!+\!1} = \frac{5\gamma_1}{k+5}$, we choose $\beta_k$ from the third condition of \eqref{eq:ama_param_cond} as $\beta_k = \frac{2\norm{\Ab}^2\tau_k^2}{(1-\tau_k^2)\gamma_{k\!+\!1}} = \frac{18\norm{\Ab}^2(k+5)}{5\gamma_1(k+1)(k+7)}$ for $k\geq 1$.
Using the value of $\tau_k$ and $\beta_k$, we need to check the second condition  $\beta_{k\!+\!1} \geq (1-\tau_k)\beta_k$ of \eqref{eq:ama_param_cond}. 
Indeed, this condition is equivalent to $2k^2 + 28k + 88 \geq 0$, which is true for all $k\geq 0$.
From the update rule of $\beta_k$, it is obvious that $\beta_k \leq \frac{18\norm{\Ab}^2}{5\gamma_1(k+1)}$.
\Eproof

\beforesubsubsec
\subsubsection{The proof of Theorem \ref{th:new_ama}: Convergence  of Algorithm \ref{alg:SAMA}}\label{apdx:th:new_ama}
\aftersubsubsec
We estimate the term $\tau_k^2\eta_k$ in \eqref{eq:ama_key_estimate2} as
\begin{align*}
\tau_k^2\eta_k = \frac{45\gamma_1}{2\norm{\Ab}^2(k\!+\!4)^2(k\!+\!5)} < \frac{45\gamma_1}{2\norm{\Ab}^2(k\!+\!4)(k\!+\!5)} - \left(1 \!-\! \tau_k\right)\frac{45\gamma_1}{2\norm{\Ab}^2(k\!+\!3)(k+4)}.
\end{align*}
Combing this estimate and \eqref{eq:ama_key_estimate2}, we get 
\begin{equation*}
G_{k\!+\!1}(\bar{\wb}^{k\!+\!1}) - \frac{45\gamma_1D_f^2}{8\norm{\Ab}^2(k+4)(k+5)} \leq (1-\tau_k)\left[G_k(\bar{\wb}^k) - \frac{45\gamma_1D_f^2}{8\norm{\Ab}^2(k+3)(k+4)}\right].
\end{equation*} 
By induction, we have $G_k(\bar{\wb}^k) - \frac{45\gamma_1D_f^2}{8\norm{\Ab}^2(k+3)(k+4)} \leq \omega_k[G_1(\bar{\wb}^1) - \frac{9\gamma_1}{32\norm{\Ab}^2}D_f^2] \leq 0$ whenever $G_1(\bar{\wb}^1) \leq \frac{3\gamma_1}{4\norm{\Ab}^2}D_f$, where $\omega_k := \prod_{i=1}^{k-1}(1-\tau_i)$.
Hence, we finally get
\begin{equation}\label{eq:ama_th3_est10}
G_{k}(\bar{\wb}^{k}) \leq \frac{45\gamma_1D_f^2}{8\norm{\Ab}^2(k+3)(k+4)}.
\end{equation}
Since $\eta_0 = \frac{\gamma_1}{2\norm{\Ab}^2}$, it satisfies the condition $5\gamma_1 > 2\eta_0\norm{\Ab}^2$ in Lemma~\ref{le:init_point}.
In addition, from Lemma~\ref{le:update_rules_ama}, we have $\beta_1 = \frac{27\norm{\Ab}^2}{20\gamma_1} > \frac{\norm{\Ab}^2}{\gamma_1}$, which satisfies the second condition in Lemma~\ref{le:init_point}.
We also note that $\beta_k \leq \frac{18\norm{\Ab}^2}{5\gamma_1(k+1)}$. 
If we take $\hat{\lbd}^0 = \boldsymbol{0}^m$, then Lemma~\ref{le:init_point} shows that $G_{\gamma_1\beta_1}(\bar{\wb}^1) \leq \frac{\eta_0}{2}D_f^2 = \frac{\gamma_1}{4\norm{\Ab}^2}D_f^2 < \frac{9\gamma_1}{32\norm{\Ab}^2}D_f^2$.
Using this estimate and \eqref{eq:ama_th3_est10} into Lemma \ref{le:optimal_bounds}, we obtain \eqref{eq:convergence_AMA}.
Finally, if we choose $\gamma_1 := \norm{\Ab}$, then we obtain the worst-case iteration-complexity of Algorithm \ref{alg:SAMA} is $\mathcal{O}(\varepsilon^{-1})$.
\Eproof

\beforesubsec
\subsection{The proof of Corollary \ref{co:ama_variant2}: Strong convexity of $g$}\label{apdx:co:ama_variant2}
\aftersubsec
First, we show that if condition \eqref{eq:ama_param_cond_scvx} hold, then \eqref{eq:ama_key_estimate_scvx} holds.
Since  $\nabla{\varphi}$ given by \eqref{eq:dual_func_i} is Lipschitz continuous with $L_{d^g_0} := \mu_g^{-1}\norm{\Ab}^2$, similar to the proof of Lemma  \ref{le:ama_key_estimate2}, we have
\begin{equation}\label{eq:ama_cor47_est1}
\Delta{G_{\beta_k}} \geq \left[ \eta_k\left(\frac{3}{4} + \frac{\tau_k}{2} - \frac{\eta_k\Vert\Ab\Vert^2}{2\mu_g}\right) - \frac{\tau_k^2}{2(1-\tau_k)\beta_k}\right]\Vert\hat{\zb}^{k\!+\!1}\Vert^2 - \frac{\tau_k^2\eta_k}{4}D_f^2,
\end{equation}
where $\Delta{G_{\beta_k}} := (1-\tau_k)G_{\beta_k}(\bar{\wb}^k) - G_{\beta_{k\!+\!1}}(\bar{\wb}^{k\!+\!1})$.
Under the condition \eqref{eq:ama_param_cond_scvx}, \eqref{eq:ama_cor47_est1} implies \eqref{eq:ama_key_estimate_scvx}.

The update rule \eqref{eq:update_params_ama_scvx} is in fact derived from \eqref{eq:ama_param_cond_scvx}. 
We finally prove the bounds \eqref{eq:convergence_AMA_scvx}.
First, we consider the product $\tau^2_k\eta_k$. By \eqref{eq:update_params_ama_scvx} we have
\begin{align*}
\tau_k^2\eta_k &= \frac{9\mu_g}{2\norm{\Ab}^2(k\!+\!4)^2} < \frac{9\mu_g}{2\norm{\Ab}^2(k\!+\!3)(k\!+\!4)} =  \frac{9\mu_g}{4\norm{\Ab}^2(k\!+\!4)} -  (1-\tau_k)\frac{9\mu_g}{4\norm{\Ab}^2(k\!+\!3)}
\end{align*}
By induction, it follows from \eqref{eq:ama_key_estimate_scvx} and this last expression that:
\begin{equation}\label{eq:ama_cor_36_est6}
G_{\beta_k}(\bar{\wb}^k) -  \frac{9\mu_gD_f^2}{16\norm{\Ab}^2(k+3)} \leq \omega_k\Big(G_{\beta_1}(\bar{\wb}^1) - \frac{9\mu_gD_f^2}{64\norm{\Ab}^2}\Big) \leq 0, 
\end{equation}
whenever $G_{\beta_1}(\bar{\wb}^1) \leq \frac{9\mu_gD_f^2}{64\norm{\Ab}^2}$. 
Since $\bar{\ub}^1$ is given by \eqref{eq:ama_init_point_scvx}, with the same argument as the proof of Lemma \ref{le:init_point}, we can show that if $\frac{1}{\beta_1} \leq \frac{5\eta_0}{2} - \frac{\norm{\Ab}^2\eta_0^2}{\mu_g}$, then $G_{\beta_1}(\bar{\wb}^1) \leq \frac{\eta_0}{4}D_f^2$.
However, from the update rule \eqref{eq:update_params_ama_scvx}, we can see that  $\eta_0 = \frac{\mu_g}{2\norm{\Ab}^2}$ and $\beta_1 = \frac{18\norm{\Ab}^2}{16\mu_g}$.
Using these quantities, we can clearly show that $\frac{1}{\beta_1} \leq \frac{5\eta_0}{2} - \frac{\norm{\Ab}^2\eta_0^2}{\mu_g} = \frac{\mu_g}{\norm{\Ab}^2}$. Moreover, $G_{\beta_1}(\bar{\wb}^1) \leq \frac{\eta_0}{4}D_f^2 < \frac{9\mu_g}{64\norm{\Ab}^2}D_f^2$. Hence, \eqref{eq:ama_cor_36_est6} holds.
Finally, it remains to use Lemma \ref{le:optimal_bounds} to obtain \eqref{eq:convergence_AMA_scvx}.
The second part in \eqref{eq:convergence_AMA_scvx2} is proved similarly. 
The estimate \eqref{eq:primal_dual_guarantee} is a direct consequence of \eqref{eq:ama_cor_36_est6}.
\Eproof

\beforesubsec
\subsection{Convergence analysis of Algorithm \ref{alg:SADMM}}
\aftersubsec
This appendix provides full proof of Lemmas and Theorems related to the convergence of Algorithm \ref{alg:SADMM}.

\beforesubsubsec
\subsubsection{The proof of Lemma \ref{le:key_estimate2_admm}: Gap reduction condition}\label{apdx:le:key_estimate2_admm}
\aftersubsubsec
We first require the following key lemma to analyze the convergence of our  \ref{eq:new_admm} scheme, whose proof is similar to \eqref{eq:key_estimate1} and we omit the details here.

\begin{lemma}\label{le:key_estimate1_admm}
Let $\bar{\lbd}^{k\!+\!1}$ be generated by  \ref{eq:new_admm}. Then, for $\lambda\in\R^n$, one has
\vspace{-0.75ex}
\begin{align*}
d_{\gamma_{k\!+\!1}}(\bar{\lbd}^{k\!+\!1}) \leq \tilde{\ell}_{\gamma_{k\!+\!1}}(\lbd) \!+\! \tfrac{1}{\eta_k}\iprods{\bar{\lbd}^{k\!+\!1} \!-\!\hat{\lbd}^k, \lbd \!-\! \hat{\lbd}^k} \!-\! \tfrac{1}{\eta_k}\norm{\hat{\lbd}^k \!-\! \bar{\lbd}^{k\!+\!1}}^2  \!+\! \tfrac{\norm{\Ab}^2}{2\gamma_{k\!+\!1}}\norm{\tilde{\lbd}^k \!-\! \bar{\lbd}^{k\!+\!1}}^2,
\vspace{-0.75ex}
\end{align*}
where $\tilde{\lbd}^k := \hat{\lbd}^k - \rho_k(\Ab\hat{\ub}^{k\!+\!1} + \Bb\hat{\vb}^k - \cb)$ and $\tilde{\ell}_{\gamma}(\lbd) := \varphi_{\gamma}(\tilde{\lambda}^k) + \iprods{\nabla{\varphi_{\gamma}}(\tilde{\lambda}^k), \lambda - \tilde{\lambda}^k} + \psi(\lambda)$.
\end{lemma}


\vspace{1ex}
Now, we can prove Lemma \ref{le:key_estimate2_admm}.
We still use the same notations as in the proof of Lemma \ref{le:ama_key_estimate2}.
In addition, let us denote by $\hat{\ub}^{*}_{k\!+\!1} := \ub^{\ast}_{\gamma_{k+1}}(\Ab^{\top}\hat{\lbd}^k)$ and $\bar{\ub}^{\ast}_{k+1} := \ub^{\ast}_{\gamma_{k+1}}(\Ab^{\top}\bar{\lbd}^k)$ given in \eqref{eq:u_ast_gamma}, $\tilde{\zb}^k := \Ab\hat{\ub}^{k\!+\!1} + \Bb\hat{\vb}^k - \cb$ and $\breve{D}_k := \Vert\Ab\hat{\ub}^{*}_{k\!+\!1} + \Bb\hat{\vb}^k - \cb\Vert$. 

First, since $\varphi_{\gamma}(\tilde{\lambda}^k) + \iprods{\nabla{\varphi_{\gamma}}(\tilde{\lambda}^k), \lambda - \tilde{\lambda}^k}  \leq \varphi_{\gamma}(\lbd)$, it follows from Lemma~\ref{le:key_estimate1_admm} that
\begin{equation}\label{eq:admm_lm51_est1a}
\begin{array}{ll}
d_{\gamma_{k\!+\!1}}(\bar{\lbd}^{k\!+\!1}) &\leq d_{\gamma_{k\!+\!1}}(\lbd) \!+\! \tfrac{1}{\eta_k}\iprods{\bar{\lbd}^{k\!+\!1} \!-\!\hat{\lbd}^k, \lbd \!-\! \hat{\lbd}^k} \!-\! \tfrac{1}{\eta_k}\norm{\hat{\lbd}^k \!-\! \bar{\lbd}^{k\!+\!1}}^2  \vspace{1ex}\\
& + \tfrac{\norm{\Ab}^2}{2\gamma_{k\!+\!1}}\norm{\tilde{\lbd}^k \!-\! \bar{\lbd}^{k\!+\!1}}^2.
\end{array}
\end{equation}
Next, using \cite[Theorem 2.1.5 (2.1.10)]{Nesterov2004} with $g^{\ast}_{\gamma}$ defined in \eqref{eq:d1_gamma} and $\lbd := (1-\tau_k)\bar{\lbd}^k + \tau_k\hat{\lbd}^k$ for any $\tau_k\in [0, 1]$, we have
\begin{equation}\label{eq:admm_lm51_est1b}
\varphi_{\gamma_{k+1}}(\lbd) \leq (1-\tau_k)\varphi_{\gamma_{k+1}}(\bar{\lbd}^k) + \tau_k\varphi_{\gamma_{k+1}}(\hat{\lbd}^k) - \frac{\tau_k(1-\tau_k)\gamma_{k+1}}{2}\Vert \hat{\ub}^{*}_{k\!+\!1} - \bar{\ub}^{*}_{k\!+\!1}\Vert^2.
\end{equation}
Since $\psi$ is convex, we also have $\psi(\lbd) \leq (1-\tau_k)\psi(\bar{\lbd}^k) + \tau_k\psi(\hat{\lbd}^k)$ and $\lbd - \hat{\lbd}^k = (1-\tau_k)\bar{\lbd}^k + \tau_k\hat{\lbd}^k - \hat{\lbd}^k = \tau_k(\hat{\lbd}^k - \lbd^{\ast}_k)$ due to \eqref{eq:add_admm_steps}.
Combining these expressions, the definition $d_{\gamma} := \varphi_{\gamma} + \psi$, \eqref{eq:admm_lm51_est1a}, and \eqref{eq:admm_lm51_est1b}, we can derive
\begin{equation}\label{eq:admm_lm51_est1}
{}\begin{array}{ll}
d_{\gamma_{k\!+\!1}}(\bar{\lambda}^{k\!+\!1}) &{\!\!}\leq (1-\tau_k)d_{\gamma_{k\!+\!1}}(\bar{\lambda}^k) \!+\! \tau_kd_{\gamma_{k\!+\!1}}(\hat{\lambda}^k) \!+\! \frac{\tau_k}{\eta_k}\iprods{\bar{\lambda}^{k\!+\!1} \!-\! \hat{\lambda}^k, \hat{\lambda}^k \!-\!  \lambda^{*}_k} \vspace{1ex}\\
& - \frac{1}{\eta_k}\Vert\bar{\lambda}^{k\!+\!1} \!-\! \hat{\lambda}^k\Vert^2 + \frac{\norm{\Ab}^2}{2\gamma_{k\!+\!1}}\Vert\bar{\lambda}^{k\!+\!1} \!-\! \tilde{\lambda}^k\Vert^2 \vspace{1ex}\\
& - (1-\tau_k)\tau_k\frac{\gamma_{k\!+\!1}}{2}\norm{\bar{\ub}^{\ast}_{k\!+\!1} - \hat{\ub}^{\ast}_{k\!+\!1}}^2.
\end{array}{\!\!}
\end{equation}
On the one hand, since $\hat{\ub}^{k+1}$ is the solution of the first convex subproblem in \ref{eq:new_admm}, using its optimality condition, we can show that
\begin{equation}\label{eq:admm_lemma51_est2}
\begin{array}{ll}
\varphi_{\gamma_{k+1}}(\hat{\lbd}^k) - \frac{\rho_k}{2}\breve{D}_k^2 &= \iprods{\hat{\lambda}^k, \Ab\hat{\ub}^{*}_{k\!+\!1}} - g(\hat{u}^{*}_{k\!+\!1}) - \gamma_{k+1}b_{\Uc}(\hat{\ub}^{*}_{k\!+\!1},\bar{\ub}^c) - \frac{\rho_k}{2}\breve{D}_k^2\vspace{1ex}\\
&{\!}\leq   \iprods{\hat{\lambda}^k,\Ab\hat{\ub}^{k\!+\!1}} - g(\hat{\ub}^{k\!+\!1}) - \frac{\rho_k}{2}\Vert\tilde{\zb}^k\Vert^2 - \gamma_{k\!+\!1}b_{\Uc}(\hat{\ub}^{k\!+\!1}, \bar{\ub}_c)\vspace{1ex}\\
&{\!} - \frac{\rho_k}{2}\Vert\Ab(\hat{\ub}^{*}_{k\!+\!1} - \hat{\ub}^{k\!+\!1})\Vert^2 - \frac{\gamma_{k+1}}{2}\Vert \hat{\ub}^{*}_{k\!+\!1} - \hat{\ub}^{k\!+\!1} \Vert^2.
\end{array}
\end{equation}
On the other hand, similar to the proof of Lemma \ref{le:ama_key_estimate2}, we can show that
\begin{align}\label{eq:admm_lemma51_est3}
\psi(\hat{\lambda}^k) &{\!\!\!}\leq  \iprods{\hat{\lambda}^k, \Bb\hat{\vb}^{k\!+\!1} \!\!-\! \cb} \!-\! h(\hat{\vb}^{k\!+\!1}) \!-\! \tfrac{\eta_k}{2}\Vert\hat{\zb}^{k\!+\!1}\Vert^2 \!-\! \tfrac{\eta_k}{2}\iprods{\hat{\zb}^{k\!+\!1}, \Ab\hat{\ub}^{k\!+\!1} \!+\! \Bb(2\hat{\vb}^{*}_k \!-\! \hat{\vb}^{k\!+\!1}) \!-\! \cb}{\!\!}\nonumber\\
&\leq \iprods{\hat{\lambda}^k, \Bb\hat{\vb}^{k\!+\!1} -\cb} - h(\hat{\vb}^{k\!+\!1}) - \tfrac{\eta_k}{2}\Vert\hat{\zb}^{k\!+\!1}\Vert^2 + \tfrac{\eta_k}{2}\norm{\hat{\zb}^{k\!+\!1}}D_k.{\!}
\end{align}
Combining \eqref{eq:admm_lemma51_est2} and \eqref{eq:admm_lemma51_est3} and noting that $d_{\gamma} := \varphi_{\gamma} + \psi$, we have
\begin{equation}\label{eq:admm_lemma51_est4}
{\!\!\!\!}\begin{array}{ll}
d_{\gamma_{k\!+\!1}}(\hat{\lambda}^k) &{\!}\leq  \iprods{\hat{\lambda}^k, \hat{\zb}^{k\!+\!1}}  - f(\hat{\xb}^{k\!+\!1}) - \frac{\eta_k}{2}\Vert\hat{\zb}^{k\!+\!1}\Vert^2
-  \frac{\rho_k}{2}\Vert\tilde{\zb}^k\Vert^2 - \gamma_{k\!+\!1}b_{\Uc}(\hat{\ub}^{k\!+\!1},\bar{\ub}_c)\vspace{1.2ex}\\ 
&{\!} - \frac{ \rho_k}{2}\Vert\Ab(\hat{\ub}^{*}_{k\!+\!1} - \hat{\ub}^{k\!+\!1})\Vert^2 - \frac{\gamma_{k+1}}{2}\Vert \hat{\ub}^{*}_{k\!+\!1} - \hat{\ub}^{k\!+\!1} \Vert^2 + \frac{\eta_k}{2}\Vert\hat{\zb}^{k\!+\!1}\Vert  D_k + \frac{\rho_k}{2}\breve{D}_k^2.
\end{array}{\!\!\!\!\!}
\end{equation}
Next, using the strong convexity of $b_{\Uc}$ with $\mu_{b_{\Uc}} = 1$, we can show that
\begin{equation}\label{eq:admm_lemma51_est5}
{}\begin{array}{ll}
\frac{\gamma_{k\!+\!1}}{2}\Vert \hat{\ub}^{*}_{k\!+\!1} - \hat{\ub}^{k\!+\!1} \Vert^2 + \gamma_{k\!+\!1}b_{\Uc}(\hat{\ub}^{k\!+\!1}, \bar{\ub}_c) \geq \frac{\gamma_{k\!+\!1}}{4}\Vert \hat{\ub}^{\ast}_{k\!+\!1} \!-\! \bar{\ub}_c \Vert^2.
\end{array}{\!\!\!}
\end{equation}
Combining \eqref{eq:admm_lm51_est1}, \eqref{eq:ama_lemma41_est2}, \eqref{eq:admm_lemma51_est4} and \eqref{eq:admm_lemma51_est5}, we can derive
\begin{equation}\label{eq:admm_lemma51_est6}
\begin{array}{ll}
d_{\gamma_{k\!+\!1}}(\bar{\lambda}^{k\!+\!1}) &{\!\!}\leq (1-\tau_k)d_{\gamma_k}(\bar{\lambda}^k) \!+\! \frac{\tau_k}{\eta_k}\iprods{\bar{\lambda}^{k\!+\!1}\!-\!\hat{\lambda}^k, \hat{\lambda}^k \!-\! \lambda^{*}_k}  \vspace{1ex}\\
&{\!\!} - \frac{1}{\eta_k}\Vert\bar{\lambda}^{k\!+\!1} \!-\! \hat{\lambda}^k\Vert^2  + \frac{\Vert\Ab\Vert^2}{2\gamma_{k\!+\!1}}\Vert\bar{\lambda}^{k\!+\!1} \!-\! \tilde{\lambda}^k\Vert^2 \vspace{1ex}\\
&{\!\!}  - \tau_kf(\hat{\xb}^{k\!+\!1}) \!+\! \tau_k\iprods{\hat{\lambda}^k, \hat{\zb}^{k\!+\!1}} \!-\! \frac{\tau_k\eta_k}{2}\Vert\hat{\zb}^{k\!+\!1}\Vert^2 \!-\!  \frac{\tau_k\rho_k}{2}\Vert\tilde{\zb}^k\Vert^2\vspace{1ex}\\
&{\!\!}  - \frac{\tau_k\gamma_{k\!+\!1}}{4}\Vert \hat{\ub}^{\ast}_{k\!+\!1} \!-\! \bar{\ub}_c \Vert^2 - (1 \!-\! \tau_k)\tau_k\frac{\gamma_{k\!+\!1}}{2}\Vert \hat{\ub}^{\ast}_{k\!+\!1} -  \bar{\ub}^{\ast}_{k\!+\!1}\Vert^2\vspace{1ex}\\
&{\!\!}  + (1\!-\!\tau_k)(\gamma_k \!-\! \gamma_{k\!+\!1})b_{\Uc}(\bar{\ub}^{\ast}_{k\!+\!1}, \bar{\ub}^c)  + \frac{\tau_k\eta_k}{2}\Vert\hat{\zb}^{k\!+\!1}\Vert D_k  +  \frac{\tau_k\rho_k}{2}\breve{D}_k^2.
\end{array}{}
\end{equation}
Let us define 
\begin{equation}\label{eq:R_k2}
\begin{array}{ll}
\hat{R}_k &:= \frac{\gamma_{k\!+\!1}}{2}(1\!-\!\tau_k)\tau_k\Vert \hat{\ub}^{\ast}_{k\!+\!1} -  \bar{\ub}^{\ast}_{k\!+\!1}\Vert^2  \!+\! \frac{\gamma_{k\!+\!1}}{4}\tau_k\Vert \hat{\ub}^{\ast}_{k\!+\!1} \!-\! \bar{\ub}_c \Vert^2 \vspace{1ex}\\
&   -  (1 - \tau_k)(\gamma_k - \gamma_{k\!+\!1})b_{\Uc}(\bar{\ub}^{\ast}_{k\!+\!1}, \bar{\ub}^c).
\end{array}
\end{equation}
From \ref{eq:new_admm}, we have $\bar{\lambda}^{k+1} - \hat{\lambda}^k = -\eta_k\hat{\zb}^{k\!+\!1}$ and $\tilde{\lambda}^k - \hat{\lambda}^k = -\rho_k\tilde{\zb}^k$. 
Plugging these expressions and \eqref{eq:R_k2} into \eqref{eq:admm_lemma51_est6} we can simplify this estimate as
\begin{equation}\label{eq:admm_lemma51_est7}
\begin{array}{ll}
d_{\gamma_{k\!+\!1}}(\bar{\lambda}^{k\!+\!1}) &{\!\!}\leq (1-\tau_k)d_{\gamma_k}(\bar{\lambda}^k) \!+\! \tau_k\iprods{\hat{\zb}^{k\!+\!1},\lambda^{*}_k}  - \tau_kf(\hat{\xb}^{k\!+\!1}) - \frac{(1+\tau_k)\eta_k}{2}\Vert\hat{\zb}^{k\!+\!1}\Vert^2\vspace{1ex}\\
&{\!\!} - \frac{1}{\eta_k}\Vert\bar{\lambda}^{k\!+\!1} \!-\! \hat{\lambda}^k\Vert^2   + \frac{\Vert\Ab\Vert^2}{2\gamma_{k\!+\!1}}\Vert\bar{\lambda}^{k\!+\!1} \!-\! \tilde{\lambda}^k\Vert^2 - \frac{\tau_k}{2\rho_k}\Vert\tilde{\lambda}^k - \hat{\lambda}^k\Vert^2  - \hat{R}_k \vspace{1ex}\\
&{\!\!} + \frac{\tau_k\eta_k}{2}\Vert\hat{\zb}^{k\!+\!1}\Vert D_k + \frac{\tau_k\rho_k}{2}\breve{D}_k^2.
\end{array}{}
\end{equation}
Using again the elementary inequality $\nu\norm{a}^2 + \kappa\norm{b}^2 \geq \frac{\nu\kappa}{\nu+\kappa}\norm{a - b}^2$, under the condition $\gamma_{k\!+\!1} \geq \norm{\Ab}^2\left(\eta_k + \frac{\rho_k}{\tau_k}\right)$ in \eqref{eq:param_conds}, we can show that
\begin{equation}\label{eq:admm_lemma51_est7b}
\frac{1}{2\eta_k}\norm{\bar{\lbd}^{k\!+\!1} - \hat{\lbd}^k}^2 + \frac{\tau_k}{2\rho_k}\norm{\tilde{\lbd}^k - \hat{\lbd}^k}^2 -  \frac{\Vert\Ab\Vert^2}{2\gamma_{k\!+\!1}}\norm{\bar{\lbd}^{k\!+\!1} - \tilde{\lbd}^k}^2 \geq 0.
\end{equation}
On the other hand, similar to the proof of Lemma \ref{le:ama_key_estimate2}, we can show that 
$\frac{\eta_k}{4}\norm{\hat{\zb}^{k\!+\!1}}^2 - \frac{\tau_k\eta_k}{2}\norm{\hat{\zb}^{k\!+\!1}}D_k \geq - \frac{\eta_k\tau_k^2}{4}D_k^2$.
Using this inequality, \eqref{eq:admm_lemma51_est7b}, and $\lambda^{*}_k = -\frac{1}{\beta_k}\bar{\zb}^k$, we can simplify \eqref{eq:admm_lemma51_est7} as
\begin{align}\label{eq:admm_lemma51_est8}
d_{\gamma_{k\!+\!1}}(\bar{\lambda}^{k\!+\!1}) &{\!\!}\leq (1-\tau_k)d_{\gamma_k}(\bar{\lambda}^k) - \tfrac{\tau_k}{\beta_k}\iprods{\hat{\zb}^{k\!+\!1},\bar{\zb}^k}  - \tau_kf(\hat{\xb}^{k\!+\!1}) - \eta_k\left(\tfrac{1}{4} + \tfrac{\tau_k}{2}\right)\Vert\hat{\zb}^{k\!+\!1}\Vert^2\nonumber\\
&{\!\!} - \hat{R}_k + \left(\tfrac{\eta_k\tau_k^2}{4}D_k^2 + \tfrac{\tau_k\rho_k}{2}\breve{D}_k^2\right).
\end{align}
Since $\beta_{k+1} \geq (1-\tau_k)\beta_k$ due to \eqref{eq:param_conds}, similar to the proof of \eqref{eq:ama_lemma41_est5b} we have
\begin{equation}\label{eq:admm_lemma51_est9}
\begin{array}{ll}
{\!\!}\Delta{G}_k &\geq (1 \!-\! \tau_k)d_{\gamma_k}(\bar{\lambda}^k) - d_{\gamma_{k\!+\!1}}(\bar{\lambda}^{k\!+\!1}) -  \tau_kf(\hat{\xb}^{k\!+\!1}) \vspace{1ex}\\
& -  \frac{\tau_k}{\beta_k}\iprods{\hat{\zb}^{k\!+\!1}, \bar{\zb}^k}  \!-\! \frac{\tau_k^2}{2(1 \!-\! \tau_k)\beta_k}\Vert\hat{\zb}^{k\!+\!1}\Vert^2.{\!\!}
\end{array}
\end{equation}
Combining \eqref{eq:admm_lemma51_est8} and \eqref{eq:admm_lemma51_est9}, we get
\begin{equation}\label{eq:admm_lemma51_est10}
\Delta{G}_k \geq \frac{1}{2}\Big[ \Big(\frac{1}{2} + \tau_k\Big)\eta_k  - \frac{\tau_k^2}{(1 \!-\! \tau_k)\beta_k}\Big]\Vert\hat{\zb}^{k\!+\!1}\Vert^2 + \hat{R}_k - \left(\frac{\eta_k\tau_k^2}{4}D_k^2 + \frac{\tau_k\rho_k}{2}\breve{D}_k^2\right).
\end{equation}
Next, we estimate $\hat{R}_k$ defined by \eqref{eq:R_k2} as follows.
We define $\bar{a}_k := \bar{\ub}^{*}_{k\!+\!1} - \bar{\ub}_c$, $\hat{a}_k := \hat{\ub}^{*}_{k\!+\!1} - \bar{\ub}_c$. 
Using $b_{\Uc}(\bar{\ub}^{\ast}_{k+1}, \bar{\ub}^c) \leq \frac{L_b}{2}\Vert \bar{\ub}^{\ast}_{k+1} - \bar{\ub}^c\Vert^2$, we can write $\hat{R}_k$ explicitly as
\begin{equation*} 
\begin{array}{ll}
\frac{2\hat{R}_k}{\gamma_{k+1}} &=  (1 - \tau_k)\tau_k\norm{\bar{a}_k -  \hat{a}_k}^2 + \frac{\tau_k}{2}\norm{\hat{a}_k}^2  - (1 - \tau_k)\big(\frac{\gamma_{k}}{\gamma_{k+ 1}} - 1\big)L_b\norm{\bar{a}_k}^2\vspace{1ex}\\
&= \tau_k\left(\frac{3}{2} -\tau_k\right)\left\Vert \hat{a}_k - \frac{(1-\tau)}{(3/2-\tau_k)}\bar{a}_k\right\Vert^2 + (1-\tau_k)\left[\frac{\tau_k}{3-2\tau_k} + \left(1- \frac{\gamma_k}{\gamma_{k+1}}\right)L_b\right]\Vert\bar{a}\Vert^2.
\end{array}
\end{equation*}
Since  $\gamma_{k+1} \geq \left(\frac{3-2\tau_k}{3 - (2-L_b^{-1})\tau_k}\right)\gamma_k$ due to \eqref{eq:param_conds}, it is easy to show that  $\hat{R}_k \geq 0$. 
In addition, by \eqref{eq:param_conds}, we also have $(1 + 2\tau_k)\eta_k - \frac{2\tau_k^2}{(1 \!-\! \tau_k)\beta_k} \geq 0$. 
Using these conditions, we can show from \eqref{eq:admm_lemma51_est10} that $\Delta{G}_k \geq - \frac{\eta_k\tau_k^2}{4}D_k^2 - \frac{\tau_k\rho_k}{2}\breve{D}_k^2 \geq -\left(\frac{\tau_k^2\eta_k}{4} + \frac{\tau_k\rho_k}{2}\right)D_f^2$, which is indeed the gap reduction condition \eqref{eq:key_estimate0}. 
\Eproof

\beforesubsubsec
\subsubsection{The proof of Lemma~\ref{le:update_admm_params}: Parameter updates}\label{apdx:le:update_admm_params}
\aftersubsubsec
Similar to the proof of Lemma \ref{le:update_rules_ama}, we can show that the optimal rate of $\set{\tau_k}$ is $\mathcal{O}(1/k)$.
From the conditions \eqref{eq:param_conds}, it is clear that if we choose $\tau_k := \frac{3}{k+4}$ then $0 < \tau_k \leq \frac{3}{4} < 1$ for $k\geq 0$.
Next, we choose $\gamma_{k+1} := \left(\frac{3-2\tau_k}{3-\tau_k}\right)\gamma_k$. 
Then $\gamma_k$ satisfies \eqref{eq:param_conds}.
Substituting $\tau_k = \frac{3}{k+4}$ into this formula we have $\gamma_{k+1} = \left(\frac{k+2}{k+3}\right)\gamma_k$.
By induction, we obtain $\gamma_{k+1} = \frac{3\gamma_1}{k+3}$.
Now, we choose $\eta_k := \frac{\gamma_{k+1}}{2\norm{\Ab}^2} = \frac{3\gamma_1}{2\norm{\Ab}^2(k+3)}$. 
Then, from the last condition of \eqref{eq:param_conds}, we choose $\rho_k := \frac{\tau_k\gamma_{k+1}}{2\norm{\Ab}^2} = \frac{9\gamma_1}{2\norm{\Ab}^2(k+3)(k+4)}$.

To derive an update for $\beta_k$, trom the third condition of \eqref{eq:param_conds} with equality, we can derive 
$\beta_k  =  \frac{2\tau_k^2}{(1-\tau_k)(1+2\tau_k)\eta_k} 
= \frac{6\norm{\Ab}^2(k+3)}{\gamma_1(k+1)(k+10)} < \frac{9\norm{\Ab}^2}{5\gamma_1(k+1)}$.
We need to check the second condition $\beta_{k+1} \geq (1-\tau_k)\beta_k$ in \eqref{eq:param_conds}. 
Indeed, we have $\beta_{k\!+\!1} = \frac{6\norm{\Ab}^2(k+4)}{\gamma_1(k\!+\!2)(k\!+\!11)} \geq (1 \!-\! \tau_k)\beta_k  = 
\frac{6\norm{\Ab}^2(k+3)}{\gamma_1(k+1)(k+10)}$, which is true for all $k\geq 0$.
Hence, the second condition of \eqref{eq:param_conds} holds. 
\Eproof

\beforesubsubsec
\subsubsection{The proof of Theorem \ref{th:convergence_of_ADMM}: Convergence  of Algorithm~\ref{alg:SADMM}}\label{apdx:th:convergence_of_ADMM}
\aftersubsubsec
First, we check the conditions of Lemma \ref{le:init_point}. 
From the update rule \eqref{eq:update_admm_params}, we have $\eta_0 = \frac{\gamma_1}{2\norm{\Ab}^2}$ and $\beta_1 = \frac{12\norm{\Ab}^2}{11\gamma_1}$.
Hence, $5\gamma_1 = 10\norm{\Ab}^2\eta_0 > 2\norm{\Ab}^2\eta_0$, which satisfies the first condition of  Lemma \ref{le:init_point}.
Now, $\frac{2\gamma_1}{(5\gamma_1-2\eta_0\norm{\Ab}^2)\eta_0} = \frac{\norm{\Ab}^2}{\gamma_1} < \frac{12\norm{\Ab}^2}{11\gamma_1} = \beta_1$. 
Hence, the second condition of  Lemma \ref{le:init_point} holds.

Next, since $\tau_k = \frac{3}{k+4}$, $\rho_k = \frac{9\gamma_1}{2\norm{\Ab}^2(k+3)(k+4)}$ and $\eta_k = \frac{3\gamma_1}{2\norm{\Ab}^2(k+3)}$, we can derive
\begin{align*}
\begin{array}{ll}
\frac{\tau_k^2\eta_k}{4} \!+\! \frac{\tau_k\rho_k}{2} &= \frac{81\gamma_1}{8\norm{\Ab}^2(k \!+\! 3)(k \!+\! 4)^2} \vspace{1ex}\\
& < \frac{81\gamma_1}{8\norm{\Ab}^2(k\!+\!3)(k\!+\!4)} - \left(1 \!-\! \tau_k\right)\frac{81\gamma_1}{8\norm{\Ab}^2(k\!+\!2)(k \!+\! 3)}.
\end{array}
 \end{align*}
 Substituting this inequality into \eqref{eq:key_estimate0} and rearrange the result we obtain
 \begin{align*} 
G_{k+1}(\bar{\wb}^{k+1}) - \frac{81\gamma_1D_f^2}{8\norm{\Ab}^2(k\!+\!3)(k\!+\!4)}  \leq (1 - \tau_k)\Big[G_k(\bar{\wb}^k) - \frac{81\gamma_1D_f^2}{8\norm{\Ab}^2(k\!+\!2)(k\!+\!3)}\Big].
 \end{align*}
 By induction, we obtain $G_k(\bar{\wb}^k) - \frac{81\gamma_1D_f^2}{8\norm{\Ab}^2(k\!+\!2)(k\!+\!3)} \leq \omega_k\Big[G_0(\bar{\wb}^0) -\frac{27\gamma_1D_f^2}{16\norm{\Ab}^2}\Big] \leq 0$ as long as $G_0(\bar{\wb}^0) \leq \frac{27\gamma_1D_f^2}{16\norm{\Ab}^2}$.
 Now using Lemma \ref{le:init_point}, we have $G_0(\bar{\wb}^0) \leq \frac{\eta_0}{4}D_f^2 = \frac{\gamma_1}{8\norm{\Ab}^2}D_f^2 < \frac{27\gamma_1D_f^2}{16\norm{\Ab}^2}$.
 Hence, $G_k(\bar{\wb}^k) \leq \frac{27\gamma_1D_f^2}{16\norm{\Ab}^2(k\!+\!2)(k\!+\!3)}$.

Finally, by using Lemma \ref{le:optimal_bounds} with $\beta_k := \frac{6\norm{\Ab}^2(k+3)}{\gamma_1(k+1)(k+10)}$ and $\beta_k  \leq \frac{9\norm{\Ab}^2}{5\gamma_1(k+1)}$, and simplifying the results, we obtain the bounds in \eqref{eq:convergence_of_admm}.
If we choose $\gamma_1 := \norm{\Ab}$ then, we obtain the worst-case iteration-complexity of Algorithm \ref{alg:SADMM} is $\mathcal{O}(\varepsilon^{-1})$.
\Eproof



\vspace{-3ex}
\bibliographystyle{abbrv}

\end{document}